\documentclass[11pt,a4paper,reqno]{amsart}
\usepackage{vmargin,color}
\usepackage[latin1]{inputenc}
\usepackage{amsmath,amsfonts,amsthm,epsfig,graphicx}
\usepackage{pstricks,pst-plot,multido}
\usepackage{caption}

\def\H{\mathcal{H}}
\def\L{\mathcal{L}}
\def\M{\mathcal{M}}

\def\N{\mathbb N}

\def\R{\mathbb R}

\def\C{\mathbf{C}}

\def\Om{\Omega}
\def\S{\Sigma}
\def\G{\Gamma}

\def\a{\alpha}
\def\b{\beta}
\def\g{\gamma}

\def\e{\varepsilon}

\def\l{\lambda}
\def\s{\sigma}

\def\om{\omega}
\def\vphi{\varphi}

\def\Lip{{\rm Lip}}

\def\Div{{\rm div}\,}

\def\dist{{\rm dist}}

\def\spt{{\rm spt}}

\def\toloc{\stackrel{{\rm loc}}{\to}}

\def\ov{\overline}
\def\pa{\partial}

\def\cc{\subset\subset}

\def\p{\mathbf{p}}
\def\q{\mathbf{q}}
\def\C{\mathbf{C}}
\def\D{\mathbf{D}}

\def\pae{\partial^{{\rm e}}}

\DeclareMathOperator*{\aplim}{ap\,lim}
\def\Gv{\Gamma^{{\rm v}}}

\def\es#1{e^{-#1^{2}/2}}
\def\esm#1{e^{-|#1|^{2}/2}}

\newtheorem{theorem}{Theorem}[section]
\newtheorem{theoremletter}{Theorem}

\newtheorem{proposition}[theorem]{Proposition}
\newtheorem{lemma}[theorem]{Lemma}
\newtheorem{corollary}[theorem]{Corollary}
\newtheorem{remark}[theorem]{Remark}
\newtheorem{example}[theorem]{Example}

\numberwithin{equation}{section}
\numberwithin{figure}{section}

\begin{document}

\title{Essential connectedness and the rigidity problem \\ for Gaussian symmetrization}

\author{F. Cagnetti}

\address{University of Sussex, Pavensey 2, Department of Mathematics, BN1 9QH, Brighton, United Kingdom}
\email{f.cagnetti@sussex.ac.uk}

\author{M. Colombo}
\address{Scuola Normale Superiore di Pisa, p.za dei Cavalieri 7, I-56126 Pisa, Italy}
\email{maria.colombo@sns.it}

\author{G. De Philippis}
\address{Institute for Applied Mathematics, University of Bonn, Endenicher Allee 60, D-53115 Bonn, Germany}
\email{guido.de.philippis@hcm.uni-bonn.de}

\author{F. Maggi}
\address{Department of Mathematics, The University of Texas at Austin,  2515 Speedway Stop C1200, Austin, Texas 78712-1202, USA}
\email{maggi@math.utexas.edu}

\maketitle

\begin{abstract}
  {\rm We provide a geometric characterization of rigidity of equality cases in Ehrhard's symmetrization inequality for Gaussian perimeter. This condition is formulated in terms of a new measure-theoretic notion of connectedness for Borel sets, inspired by Federer's definition of indecomposable current.}
\end{abstract}


\section{Introduction}

\subsection{Overview} Symmetrization inequalities are among the most basic tools used in the Calculus of Variations. The study of their equality cases plays a fundamental role in the explicit characterization of minimizers, thus in the computation of optimal constants in geometric and functional inequalities. Although it is usually easy to derive useful necessary conditions for equality cases, the analysis of {\it rigidity of equality cases} (that is, the situation when every set realizing equality in the given symmetrization inequality turns out to be symmetric) is a much subtler issue. Two deep results that provide {\it sufficient conditions} for the rigidity of equality cases are Brothers-Ziemer theorem concerning Schwartz's symmetrization inequality for the Dirichlet-type integral functionals \cite{brothersziemer}, and Chleb{\'{\i}}k-Cianchi-Fusco theorem, concerning Steiner's symmetrization inequality for distributional perimeter \cite{ChlebikCianchiFuscoAnnals05} (see \cite{barchiesicagnettifusco} for an extension of this last result to higher dimensional Steiner's symmetrization). In this paper we introduce a new point of view on rigidity of equality cases, that will allow us to provide {\it characterizations} of rigidity (rather than merely sufficient conditions) in various situations.

We address the case of Ehrhard's symmetrization inequality for Gaussian perimeter. Ehrhard's symmetrization is a powerful device in the analysis of geometric variational problems in the Gauss space, the versatility of  which is well-known in Probability Theory. Rigidity of equality cases for Ehrhard's inequality is an open problem, even at the level of finding sufficient conditions for rigidity. We shall completely solve the rigidity problem, by providing a geometric characterization of rigidity of equality cases. This characterization is formulated in terms of a measure-theoretic notion of connectedness, meaningful in the very general context of Borel sets, and inspired by the notion of indecomposable current adopted in Geometric Measure Theory; see \cite[4.2.25]{FedererBOOK}. Moreover, as we shall explain later on, the ideas and techniques developed here are not specific to the Gaussian setting, and open the possibility to obtain similar results in other frameworks.

The rest of this introduction is organized as follows. In section \ref{section gaussian perimeter} we introduce Gaussian perimeter, together with the Gaussian isoperimetric problem. This important variational problem motivates the notion of Ehrhard's symmetrization, presented in section \ref{section ehrhard}. In sections \ref{section rigidity ehrhard examples} and \ref{section connectedness} we introduce, respectively, the rigidity problem for Ehrhard's inequality, and the measure-theoretic notion of connectedness we shall exploit in its solution. In section \ref{section characterization theorem intro} we state our main result, Theorem \ref{thm gauss}, together with its proper reformulation in the planar setting. Finally, in section \ref{section confronto}, we quickly illustrate the application of our methods to Steiner's symmetrization inequality, referring to the forthcoming paper \cite{ccdpmSTEINER} for a complete discussion of this last problem.

\subsection{Gaussian perimeter and the Gaussian isoperimetric problem}\label{section gaussian perimeter} We introduce our setting. Given a Lebesgue measurable set $E\subset\R^n$, we define its Gaussian volume as
\[
\g_n(E)=\frac1{(2\pi)^{n/2}}\int_E\,\esm{x}\,dx\,.
\]
If $n\ge k\ge 1$, the $k$-dimensional Gaussian-Hausdorff measure of a Borel set $S\subset\R^n$ is
\begin{eqnarray*}
  \H^k_\g(S)=\frac1{(2\pi)^{k/2}}\int_S\,\esm{x}\,d\H^k(x)\,,
\end{eqnarray*}
where $\H^k$ denotes the $k$-dimensional Hausdorff measure on $\R^n$. (In this way, $\g_n=\H^n_\g$ and $\H^k_\g(S)=1$ whenever $S$ is a $k$-dimensional plane containing the origin.) The Gaussian perimeter of an open set $E$ with Lipschitz boundary is then defined as
\begin{equation}
  \label{gaussian perimeter}
  P_\g(E)=\H^{n-1}_\g(\pa E)=\frac1{(2\pi)^{(n-1)/2}}\int_{\pa E}\,\esm{x}\,d\H^{n-1}(x)\,.
\end{equation}
The most basic geometric variational problem in the Gauss space is, of course, the {\it Gaussian isoperimetric problem}, which consists in the minimization of Gaussian perimeter at fixed Gaussian volume. As it turns out, (the only) isoperimetric sets are half-spaces. The Gaussian isoperimetric theorem can be translated into a geometric inequality, with a characterization of equality cases. Indeed, if we define $\Phi:\R\cup\{\pm\infty\}\to[0,1]$ and $\Psi=\Phi^{-1}:[0,1]\to\R\cup\{\pm\infty\}$ by setting
\begin{eqnarray}
  \label{definition of Phi}
\Phi(t)=\frac1{\sqrt{2\pi}}\int_t^\infty\es{s}\,ds\,,\qquad t\in\R\cup\{\pm\infty\}\,,
\end{eqnarray}
then $\Phi(t)$ is the Gaussian volume of an half-space lying at ``signed distance'' $t$ from the origin (more precisely, $\Phi(t)=\g_n(\{x_1>t\})$ for every $t\in\R$). It is thus clear that, given $\l\in(0,1)$, $e^{-\Psi(\l)^2/2}$ is the Gaussian perimeter of any half-space of Gaussian volume $\l$, and thus the {\it Gaussian isoperimetric inequality} takes the form
\begin{equation}
  \label{iso gaussiana}
  P_\g(E)\ge \es{\Psi(\g_n(E))}\,,
\end{equation}
with equality if and only if, up to rotations keeping the origin fixed, $E$ is an half-space with the suitable Gaussian volume, that is
\[
E=\Big\{x\in\R^n:x_n>\Psi(\g_n(E))\Big\}\,.
\]
Inequality \eqref{iso gaussiana} was first proved by Borell \cite{BorellGAUSSIAN} and by Sudakov and Cirel'son \cite{sudakovGAUSSIAN}. Alternative proofs, either of probabilistic \cite{bakryledoux,bobkov,ledoux,barthe} or geometric \cite{ehrhard83,ehrhard84,ehrhard86} nature, have been proposed during the years, although the characterization of equality cases has been obtained only recently, by probabilistic methods, by Carlen and Kerce \cite{carlenkerce}. Finally, a characterization of equality cases, and a stability inequality with sharp decay rate, were obtained in \cite{cianchifuscomaggipratelliGAUSS} building on the symmetrization methods introduced by Ehrhard in \cite{ehrhard83}. In passing, let us mention that the study of stability issues for Gaussian isoperimetry still poses some difficult questions; see \cite{mossel} for some recent progresses in this direction.

Let us notice that the natural domain of validity of the Gaussian isoperimetric inequality, and, in fact, of Ehrhard's symmetrization technique, is much broader than what we have explained so far. Indeed, Gaussian perimeter can be defined for every Lebesgue measurable set $E\subset\R^n$ by setting
\[
P_\g(E)=\H^{n-1}_\g(\pae E)\in[0,\infty]\,.
\]
We recall that the {\it essential boundary} $\pae E$ of $E$ is defined as
\[
\pae E=\R^n\setminus\Big(E^{(0)}\cup E^{(1)}\Big)\,,
\]
where, given $t\in[0,1]$, $E^{(t)}$ denotes the set of points of density $t$ of $E$,
\[
E^{(t)}=\Big\{x\in\R^n:\lim_{r\to 0^+}\frac{\H^n(E\cap B(x,r))}{\om_n\,r^n}=t\Big\}\,,
\]
and $\om_n$ is the volume of the Euclidean unit ball of $\R^n$. If $E$ is an open set with Lipschitz boundary, then we trivially have $\pae E=\pa E$, and thus this new definition of $P_\g(E)$ provides a coherent extension of \eqref{gaussian perimeter}. In general, if $P_\g(E)<\infty$, then $E$ is a set of locally finite perimeter, and in that case $P_\g(E)=\H^{n-1}_\g(\pa^*E)$, where $\pa^*E$ denotes the {\it reduced boundary} of $E$; see section \ref{section sofp} for the terminology introduced here. (More generally, $E$ is of locally finite perimeter if and only if $E$ is of locally finite Gaussian perimeter, that is, if $\H^{n-1}_\g(K\cap\pae E)<\infty$ for every compact set $K\subset\R^n$.) Finally, we notice that, with these definitions in force, inequality \eqref{iso gaussiana} holds true for every Lebesgue measurable set $E\subset\R^n$, and equality holds if and only if, up to rotations around the origin, $E$ is $\H^n$-equivalent to the half-space $\{x\in\R^n:x_n>\Psi(\g_n(E))\}$.

\subsection{Ehrhard's symmetrization}\label{section ehrhard} Ehrhard's approach \cite{ehrhard83,ehrhard84,ehrhard86} to the Gaussian isoperimetric inequality is based on a symmetrization procedure that is the natural analogous in the Gaussian setting of Steiner's symmetrization. The definition goes as follows. We decompose $\R^n$, $n\ge 2$, as the Cartesian product $\R^{n-1}\times\R$, denoting by $\p:\R^n\to\R^{n-1}$ and $\q:\R^n\to\R$ the horizontal and vertical projections, so that $x=(\p x,\q x)$, $\p x=(x_1,...,x_{n-1})$, and $\q x=x_n$ for every $x\in\R^n$. Given a set $E\subset\R^n$, we denote by $E_z$ its vertical section with respect to $z\in\R^{n-1}$, that is, we set
\begin{equation}
  \label{vertical section}
E_z=\Big\{t\in\R:(z,t)\in E\Big\}\,,\qquad z\in\R^{n-1}\,.
\end{equation}
Given a Lebesgue measurable function $v:\R^{n-1}\to[0,1]$, we say that $E$ is $v$-distributed provided $\H^1_\g(E_z)=v(z)$ for $\H^{n-1}$-a.e. $z\in\R^{n-1}$, and we set
\begin{equation}
  \label{Fv gaussiano}
  F[v]=\Big\{x\in\R^n:\q x>\Psi(v(\p x))\Big\}\,,
\end{equation}
for the $v$-distributed set whose vertical sections are positive half-lines in the $x_n$-direction. If $E$ is a $v$-distributed set, then the {\it Ehrhard symmetral} $E^s$ of $E$ is defined as
\[
E^s=F[v]\,;
\]
see Figure \ref{fig ehrhard}.
\begin{figure}
  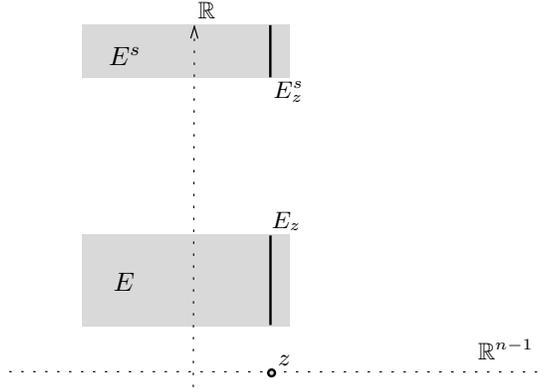\caption{{\small Ehrhard's symmetrization amounts in replacing the vertical sections of a set with vertical half-lines with same Gaussian length and positive orientation. Note that, in this picture, the non-trivial vertical sections $E_z$ of $E$ are constantly equal to a same segment. The corresponding sections $E^s_{z}$ of $E^s$ are thus constantly equal to the half-line of Gaussian length $\H^1_\g(E_z)$.}}\label{fig ehrhard}
\end{figure}
By Fubini's theorem, Gaussian volume is preserved under Ehrhard's symmetrization, that is, $\g_n(E)=\g_n(E^s)$. At the same time, Gaussian perimeter is decreased under Ehrhard's symmetrization. Precisely, if there exists a $v$-distributed set of finite Gaussian perimeter $E$, then $F[v]$ is of locally finite perimeter, and {\it Ehrhard's inequality}
\begin{equation}
  \label{ehrhard inequality}
  P_\g(E)\ge P_\g(F[v])\,,
\end{equation}
holds true. A proof of these facts based on the coarea formula is presented in \cite[Section 4.1]{cianchifuscomaggipratelliGAUSS}. This approach also leads to the following theorem concerning equality cases, that will play an important role in the sequel. (Here, $\nu_E$ denotes the measure-theoretic outer unit normal to a set of locally finite perimeter $E$; see section \ref{section sofp}.)

\begin{theoremletter}\label{thm cfmp1}
  If $E\subset\R^n$ is a set of locally finite perimeter with $P_\g(E)=P_\g(E^s)$, then
  \begin{equation}\label{necessary conditions gauss}
    \mbox{$E_z$ is $\H^1$-equivalent to a half-line for $\H^{n-1}$-a.e. $z\in\R^{n-1}$.}
  \end{equation}
  Moreover, if $E$ satisfies \eqref{necessary conditions gauss}, and $\pa^*E$ has no ``vertical parts'', that is, if
  \begin{equation}
    \label{no vertical parts on G}
      \H^{n-1}\Big(\Big\{x\in\pa^*E:\q\nu_{E}(x)=0\Big\}\Big)=0\,,
  \end{equation}
  then $P_\g(E)=P_\g(E^s)$.
\end{theoremletter}

\subsection{The rigidity problem for Ehrhard's inequality}\label{section rigidity ehrhard examples} We now turn to the rigidity problem related to the Ehrhard inequality. Given $v:\R^{n-1}\to[0,1]$ such that
\[
\M(v)=\Big\{E\subset\R^n:\mbox{$E$ is $v$-distributed and $P_\g(E)=P_\g(F[v])<\infty$}\Big\}\,,
\]
is non-empty, we ask for necessary and sufficient conditions for having that
\begin{equation}
  \label{rigidity gauss}
  \mbox{$E\in\M(v)$ if and only if either $\H^n(E\Delta F[v])=0$ or $\H^n(E\Delta \,g(F[v]))=0$}\,,
\end{equation}
where $g:\R^n\to\R^n$ denotes the reflection with respect to $\R^{n-1}$, that is
\[
g(x)=(\p x,-\q x)\,,\qquad x\in\R^n\,.
\]
Simple examples show that the rigidity condition \eqref{rigidity gauss} may fail if we allow $v$ to take the values $0$ or $1$
\begin{figure}
  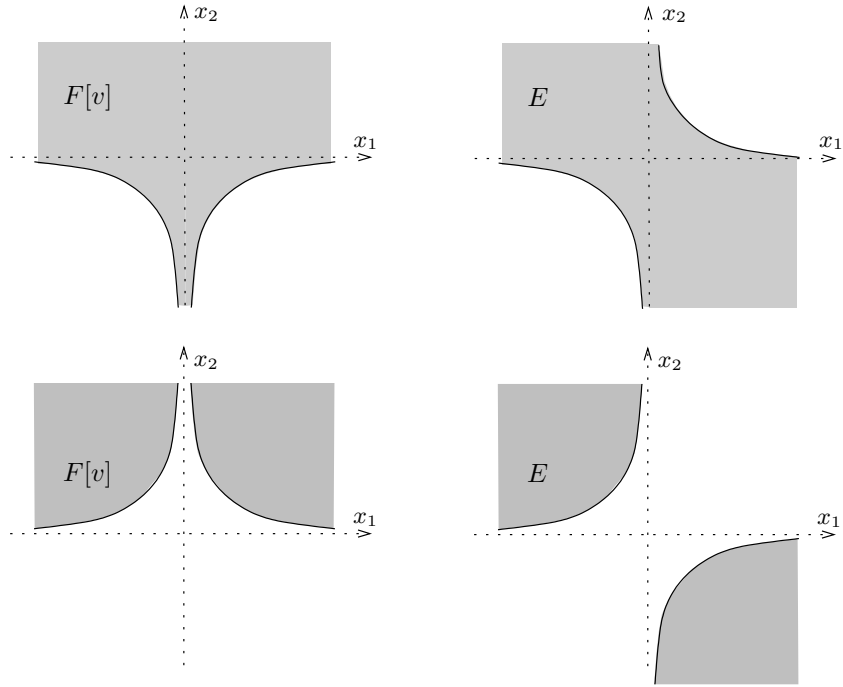\caption{{\small In the first example (two top pictures), the function $v:\R\to[0,1]$ takes the value $1$ at the origin. The correspoding set $F[v]$ is connected and there exists $E\in\M(v)$ such that $\H^2(E\Delta F)=\H^2(E\Delta g(F))=\infty$. In the second example (two bottom pictures), we observe the same features in the case of a function $v$ that takes the value $0$ at the origin.}}\label{fig examples}
\end{figure}
(see Figure \ref{fig examples}) and suggest that a reasonable sufficient condition for rigidity could amount in ruling out this possibility. At the same time, $v$ may take the values $0$ and/or $1$ and still rigidity may hold:
\begin{figure}
  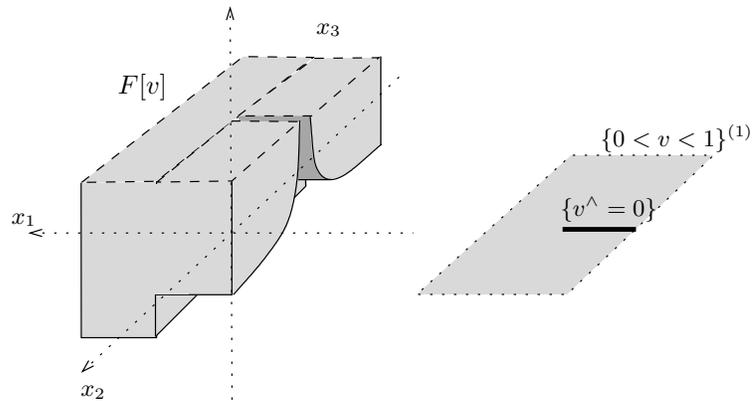\caption{{\small In this example, $\{v^\wedge=0\}$ is a segment lying inside $\{0<v<1\}^{(1)}$. Nevertheless, we have rigidity of equality cases, as a vertical reflection of $F[v]$ on any proper non-empty subset of $\{0<v<1\}$ will create extra Gaussian perimeter.}}\label{fig 01}
\end{figure}
an example is depicted in Figure \ref{fig 01}. Thus, this plausible sufficient condition would be far from being also necessary. As it turns out, one needs to introduce some proper notions of connectedness in order to formulate conditions that effectively characterize rigidity.

Before entering into this, let us notice how the need for working in a measure-theoretic framework arises naturally in here. Indeed, if $w=v$ $\H^{n-1}$-a.e. on $\R^{n-1}$, then $F[v]$ and $F[w]$ are $\H^n$-equivalent (thus $P_\g(F[v])=P_\g(F[w])\in[0,\infty]$), a set $E\subset\R^n$ is $v$-distributed if and only if it is $w$-distributed, and $\M(v)=\M(w)$. In particular, a condition like ``$v$ takes the value $0$ or $1$ on a given set $S$'' has no meaning in our problem if $\H^{n-1}(S)=0$. We shall rule out these ambiguities by exploiting the notions of approximate upper and lower limits of a Lebesgue measurable function $f:\R^m\to\R$. Precisely, the {\it approximate upper limit} $f^\vee(x)$ and the {\it approximate lower limit} $f^\wedge(x)$ of $f$ at $x\in\R^m$ are defined by setting
\begin{eqnarray}
  \label{def fvee}
  f^\vee(x)=\inf\Big\{t\in\R: x\in \{f>t\}^{(0)}\Big\}\,,
  \\
  \label{def fwedge}
  f^\wedge(x)=\sup\Big\{t\in\R: x\in \{f<t\}^{(0)}\Big\}\,.
\end{eqnarray}
In this way, $f^\vee$ and $f^\wedge$ are defined at {\it every} point of $\R^m$, with values in $\R\cup\{\pm\infty\}$, in such a way that, if $f_1=f_2$ $\H^m$-a.e. on $\R^m$, then $f_1^\vee=f_2^\vee$ and $f_1^\wedge=f_2^\wedge$ {\it everywhere} on $\R^m$. Moreover, $f^\vee$ and $f^\wedge$ turn out to be both Borel functions on $\R^m$; see section \ref{section approximate limits}.

\subsection{A measure-theoretic notion of connectedness}\label{section connectedness} Given a open set $G$ and an hypersurface $K$ in $\R^m$, the intuitive idea of what does it mean for $K$ to disconnect $G$ is pretty clear: one simply expects $K$ to be the relative boundary inside $G$ of two non-trivial, disjoint open sets $G_+$ and $G_-$ such that $G_+\cup G_-=G$. In this section, we precisely define what it means for a Borel set $K\subset\R^m$ to ``essentially'' disconnect a Borel set $G\subset\R^m$, in such a way this definition is stable under modifications of $K$ by $\H^{m-1}$-negligible sets, and of $G$ by $\H^m$-negligible sets.

In order to introduce our definition, let us first recall the measure-theoretic notion of connectedness used in the theory of sets of finite perimeter. A set of finite perimeter $G\subset\R^m$ is indecomposable (see \cite[Definition 2.11]{dolzmannmu} or \cite[Section 4]{ambrosiocaselles}), if for every non-trivial partition of $G$ into sets of finite perimeter $\{G_+,G_-\}$ modulo $\H^m$,
\begin{equation}
  \label{non trivial borel partition}
  \H^m(G_+\cap G_-)=0\,,\qquad  \H^m(G\Delta(G_+\cup G_-))=0\,,\qquad \H^m(G_+)\,\H^m(G_-)>0\,,
\end{equation}
we have that $P(G)<P(G_+)+P(G_-)$, where $P(G)=\H^{m-1}(\pa^*G)=\H^{m-1}(\pae G)$. (The indecomposability of $G$ in this sense is equivalent to the indecomposability in the sense of \cite[4.2.25]{FedererBOOK} of the $m$-dimensional integer current on $\R^m$ canonically associated to $G$.)  More generally, we can say that a set of {\it locally} finite perimeter $G\subset\R^m$ is indecomposable if there exists $r_0>0$ such that $P(G;B_r)<P(G_+;B_r)+P(G_-;B_r)$ for every $r>r_0$ and for every non-trivial partition of $G$ into sets of locally finite perimeter $\{G_+,G_-\}$. Indecomposability plays for sets of finite perimeter the same role that connectedness plays for open sets; see, for example, the various results supporting this intuition collected in \cite[Section 4]{ambrosiocaselles}. At variance with topological connectedness, indecomposability has however the following important stability property: if $G_1$ is an indecomposable set and $G_2$ is $\H^m$-equivalent to $G_1$, then $G_2$ is an indecomposable set too.

We now want to extend the notion of indecomposability to arbitrary Borel sets. Indeed, a pretty obvious necessary condition for rigidity in Ehrhard's inequality should be the ``connectedness'' of $\{0<v<1\}$. Of course, for the reasons explained so far, topological connectedness is not suitable here. Moreover, the Borel set $\{0<v<1\}$ defined by $v\in BV_{loc}(\R^{n-1};[0,1])$ may fail to be of locally finite perimeter (see Example \ref{example G per finito}), and in that case we may not exploit indecomposability. Finally, we shall in fact need to give a precise meaning to the idea that a Borel set ``disconnects'' another Borel set. This is achieved as follows. Given two Borel sets $K$ and $G$ in $\R^m$, $m\ge 1$, we say that $K$ {\it essentially disconnects} $G$ if there exists a non-trivial {\it Borel} partition $\{G_+,G_-\}$ of $G$ modulo $\H^m$ with
\begin{equation}
  \label{K disconnects G}
  \H^{m-1}\Big(\Big(G^{(1)}\cap\pae  G_+\cap\pae G_-\Big)\setminus K\Big)=0\,.
\end{equation}
Of course, we say that $K$ {\it does not essentially disconnect} $G$ if for every non-trivial Borel partition $\{G_+,G_-\}$ of $G$ modulo $\H^m$  we have
\begin{equation}
  \label{K does not disconnect G}
  \H^{m-1}\Big(\Big(G^{(1)}\cap\pae  G_+\cap\pae G_-\Big)\setminus K \Big)>0\,.
\end{equation}
Finally, we say that $G$ is {\it essentially connected} if $\emptyset$ does not essentially disconnect $G$. An example is depicted in
\begin{figure}
  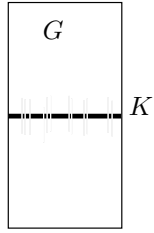\caption{{\small If $G=[0,1]\times[-1,1]\subset\R^2$ and $K\subset \ell=[0,1]\times\{0\}$, then $K$ essentially disconnects $G$ if and only if $\H^1(\ell\setminus K)=0$. Thus, the rational numbers in $[0,1]$ do not essentially disconnect $G$, while the irrational numbers in $[0,1]$ essentially disconnect $G$.}}\label{fig essentially}
\end{figure}
Figure \ref{fig essentially}.

\begin{remark}\label{remark essential connected}
  {\rm If $\H^m(G_1\Delta G_2)=0$, then $G_1^{(1)}=G_2^{(1)}$: thus, $K$ essentially disconnects $G_1$ if and only if $K$ essentially disconnects $G_2$. Similarly, if $\H^{m-1}(K_1\Delta K_2)=0$, then $K_1$ essentially disconnects $G$ if and only if $K_2$ essentially disconnects $G$.}
\end{remark}

\begin{remark}\label{remark ess equal ind}
  {\rm We shall prove in Remark \ref{remark indecomposable 2} that a set of locally finite perimeter $G\subset\R^m$ is indecomposable if and only if
  $\H^{m-1}(G^{(1)}\cap\pae G_+\cap\pae G_-)>0$ for every non-trivial {\it Borel} partition $\{G_+,G_-\}$ of $G$ modulo $\H^m$. Therefore, a set of locally finite perimeter is indecomposable if and only if it is essentially connected. At the same, the notion of essential connectedness makes sense on arbitrary Borel sets. Actually, by replacing $G^{(1)}$ with $(\R^m\setminus G)^{(0)}$ in the definition of $\pae G$, we define a notion of connectedness that should retain reasonable properties even when $G$ is a non-necessarily measurable set in $\R^m$.}
\end{remark}

\subsection{Characterizations of rigidity for Ehrhard's inequality}\label{section characterization theorem intro} We are finally into the position of stating our characterization of rigidity of equality cases in Ehrhard's inequality.

\begin{theorem}\label{thm gauss}
  If $v:\R^{n-1}\to[0,1]$ is a Lebesgue measurable function with $P_\g(F[v])<\infty$, then the following two statements are equivalent:
  \begin{enumerate}
    \item[(i)] if $E\in\M(v)$, then either $\H^n(E\Delta F[v])=0$, or $\H^n(E\Delta g(F[v]))=0$;
    \item[(ii)] the set $\{v^\wedge=0\}\cup\{v^\vee=1\}$ does not essentially disconnect $\{0<v<1\}$.
  \end{enumerate}
\end{theorem}

\begin{remark}\label{remark approx limits}
  {\rm If $v=w$ $\H^{n-1}$-a.e. on $\R^{n-1}$, then $v^\vee=w^\vee$, and $v^\wedge=w^\wedge$. In particular, the characterization of rigidity (ii) is independent of the considered representative of $v$.}
\end{remark}

\begin{remark}
  {\rm The assumption $P_\g(F[v])<\infty$ is of course the minimal hypothesis under which it makes sense to consider the rigidity problem. As we shall see in Proposition \ref{prop f GBV}, it implies a very minimal amount of regularity on $v$. Precisely, it implies that the Lebesgue measurable function $\Psi\circ v:\R^{n-1}\to\R\cup\{\pm\infty\}$ is an extended real valued function of generalized bounded variation; see section \ref{subsection epigraphs}.}
\end{remark}

Despite the geometric clarity of the characterization of rigidity presented in Theorem \ref{thm gauss}, its proof is actually quite delicate. We shall explain the reasons for this in the course of its proof, that is presented in section \ref{section gauss}. For the moment, let us just mention the following reformulation of Theorem \ref{thm gauss} in the planar case $n=2$.

\begin{theorem}\label{thm gauss R2}
  If $v:\R^{n-1}\to[0,1]$ is a Lebesgue measurable function with $P_\g(F[v])<\infty$, then the following two statements are equivalent:  \begin{enumerate}
    \item[(i)] if $E\in\M(v)$, then either $\H^2(E\Delta F[v])=0$, or $\H^2(E\Delta g(F[v]))=0$;
    \item[(ii)] $\{0<v<1\}$ is $\H^1$-equivalent to an open interval $I$, with $v^\wedge>0$ and $v^\vee<1$ on $I$.
  \end{enumerate}
\end{theorem}

\begin{remark}\label{remark sufficient}
  {\rm A natural problem is that of characterizing rigidity, or otherwise providing sufficient conditions for rigidity, in terms of indecomposability properties of $F[v]$. As shown by the examples in Figure \ref{fig examples}, it is not enough to ask that either $F[v]$ or $\R^n\setminus F[v]$ are indecomposable sets. As it turns out, if we are in the planar case, and we ask that {\it both} $F[v]$ and $\R^n\setminus F[v]$ are indecomposable sets, then rigidity holds; see Theorem \ref{thm gino}. This last condition is not necessary for rigidity in the planar case, see Figure \ref{fig spikes},
\begin{figure}
    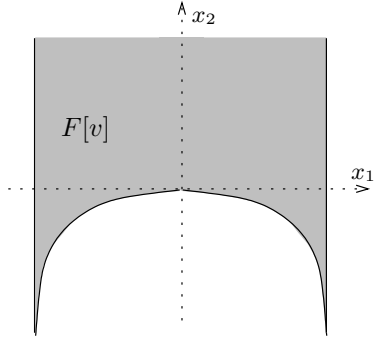\caption{{\small Asking that both $F[v]$ and $\R^n\setminus F[v]$ are indecomposable is a sufficient condition for rigidity in $\R^n$ when $n=2$, although it is not a necessary one, as this example shows.}}\label{fig spikes}
\end{figure}
and, in fact, it is not even sufficient for rigidity in $\R^n$ when $n\ge 3$;
  \begin{figure}
  \begin{center}
    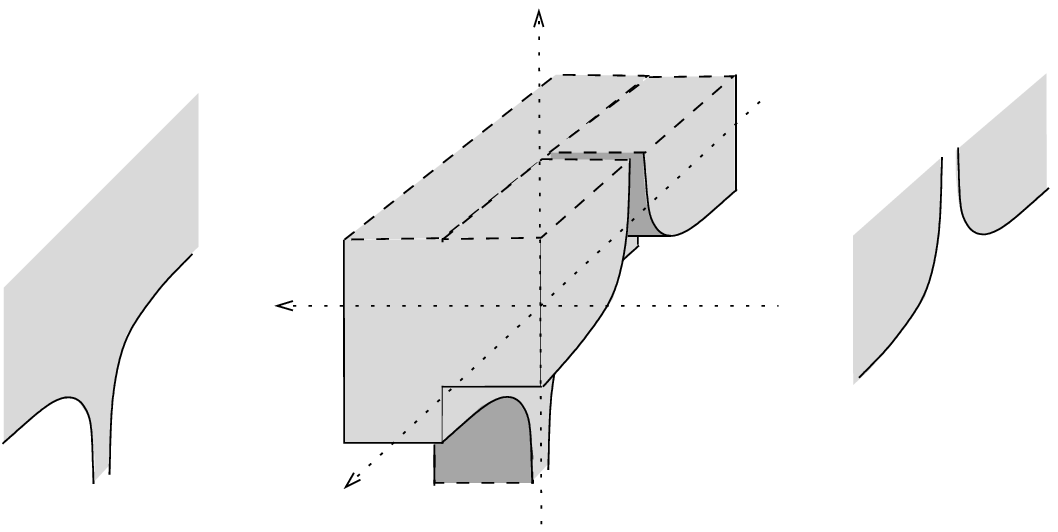
    \captionsetup{singlelinecheck=off}
    \caption[.]{{\small It may happen that both $F[v]$ and $\R^3\setminus F[v]$ are indecomposable, but rigidity fails. An example of this situation is obtained by setting
    \[
    F[v]=\Big\{x\in\R^3: 0<x_1<1\,,|x_2|<1\,,x_3>-\frac1{|x_2|}\Big\}\hspace{3cm}
    \]
    \[
    \hspace{1cm}\cup\Big\{x\in\R^3: -1<x_1<0\,,|x_2|<1\,,x_3>\frac1{|x_2|}\Big\}\,.  \hspace{3cm}
    \]
    Notice that the section $F[v]\cap\{x\in\R^3:x_1=t\}$ for $t\in(0,1)$ (depicted on the left) is an epigraph defined by two ``negative'' equilateral hyperbolas, while the section $F[v]\cap\{x\in\R^3:x_1=t\}$ for $t\in(-1,0)$ (depicted on the right) is an epigraph defined by two ``positive'' equilateral hyperbolas. Also, $\{x\in\R^2:-1<x_1<0\,,x_2=0\}\subset \{v^\wedge=0\}$ and $\{v^\vee=1\}=\{x\in\R^2:0<x_1<1\,,x_2=0\}$, so that $\{v^\wedge=0\}\cup\{v^\vee=1\}$ essentially disconnects $\{0<v<1\}=(-1,1)\times(-1,1)$, and by Theorem \ref{thm gauss} regularity fails. Indeed, the set $E$ defined by a vertical reflection of the part of $F[v]$ above $x_2>0$,
    \[
    E=\Big\{x\in F[v]:x_2<0\Big\}\cup\Big\{x\in\R^3:g(x)\in F[v]\,,x_2>0\Big\}\,,\hspace{2.8cm}
    \]
    is such that $\H^3(E\Delta F[v])>0$, $\H^3(E\Delta g(F[v]))>0$, and $P_\g(E)=P_\g(F[v])$. We also notice that condition \eqref{condition t} does not hold true in this example.}}\label{fig mistico}
  \end{center}
  \end{figure}
  see Figure \ref{fig mistico}. A sufficient condition for rigidity in $\R^n$, $n\ge 3$, is obtained by asking the existence of $\e>0$ such that
  \begin{equation}
    \label{condition t}
    \mbox{$F[v]\cap\Big(\{t<v<1-t\}\times\R\Big)$ is indecomposable for a.e. $t<\e$}\,;
  \end{equation}
  see Theorem \ref{thm pino}. However, not even this last condition is necessary for rigidity in $\R^n$: for an example in the planar case,
  \begin{figure}
    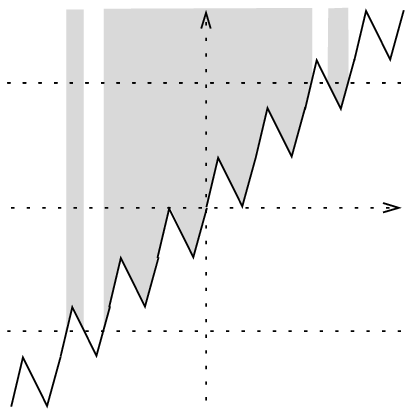\caption{{\small A planar epigraph such that rigidity holds true but condition \eqref{condition t} fails. The grey shaded area corresponds, for a generic $t\in(0,1)$, to the set $F[v]\cap(\{t<v<1-t\}\times\R)$, which turns out to be disconnected.}}\label{fig maria3}
  \end{figure}
  see Figure \ref{fig maria3}. In this case, \eqref{condition t} fails for every $t\in(0,1)$, but, of course, rigidity holds true. In conclusion, it seems not possible to achieve a characterization of rigidity in terms of indecomposability properties of $F[v]$ and related sets. At the same time, it is natural to guess that a characterization of rigidity in terms of essential connectedness should be expressed by the requirement that
  \[
  \mbox{$(\{v^\wedge=0\}\cup\{v^\vee=1\})\times\R$ does not essentially disconnect $F[v]$.}
  \]
  Although we shall not pursue this last direction here, in section \ref{section hey} we shall provide proofs of the above stated sufficient conditions for rigidity, see Theorem \ref{thm pino} and Theorem \ref{thm gino}.}
\end{remark}

\subsection{An outlook on Steiner's symmetrization inequality}\label{section confronto} With the aim to put the results and methods of this paper into the right perspective, we now present a quick overview on their applications to the study of rigidity of equality cases in Steiner's symmetrization inequality. Given a Lebesgue measurable function $v:\R^{n-1}\to[0,\infty]$ and a Lebesgue measurable set $E\subset\R^n$, at variance with the notation used in the rest of the paper, let us now say that $E$ is $v$-distributed if $\H^1(E_z)=v(z)$ for $\H^{n-1}$-a.e. $z\in\R^{n-1}$ (recall that $E_z$ denotes the vertical section of $E$, see \eqref{vertical section}), and let us set
\[
F[v]=\Big\{x\in\R^n:|\q x|<\frac{v(\p x)}2\Big\}\,,
\]
for the $v$-distributed set whose vertical sections are segments centered at height $x_n=0$. By definition, $F[v]$ is the {\it Steiner's symmetral} $E^s$ of $E$, and by Fubini's theorem, $\H^n(E)=\H^n(F[v])$. Moreover, $F[v]$ is of finite perimeter and volume if and only if $v\in BV(\R^{n-1})$ with $\H^{n-1}(\{v>0\})<\infty$. In this case, Steiner's inequality ensures that
\begin{equation}\label{steiner inequality}
P(E)\ge P(F[v])\,,
\end{equation}
whenever $E$ is a $v$-distributed set (with $P(E)=\H^{n-1}(\pae E)$). In analogy with the notation used in the Gaussian case, we set
\[
\M(v)=\Big\{E\subset\R^n:\mbox{$E$ $v$-distributed and $P(E)=P(F[v])$}\Big\}\,,
\]
so that rigidity of equality cases in \eqref{steiner inequality} amounts to say that $E\in\M(v)$ if and only if $E$ is $\H^n$-equivalent to $t\,e_n+F[v]$ for some $t\in\R$. Simple examples show that we cannot always expect rigidity of equality cases when $\pa^*F[v]$ has vertical parts, or when the length of the sections of $F[v]$ vanishes inside the projection of $F[v]$;
  \begin{figure}
  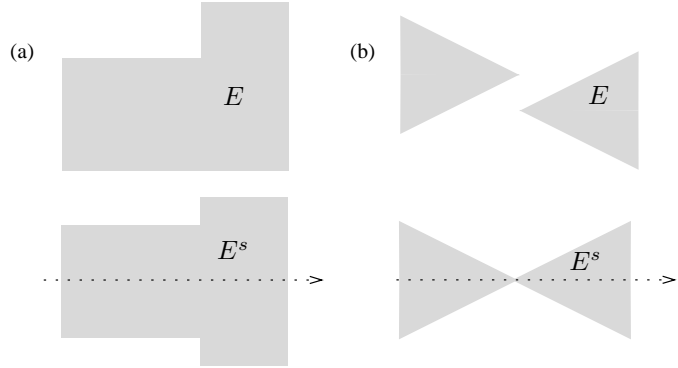\caption{{\small (a) In this case, $\pa^*E^s$ has vertical parts and rigidity fails; (b) In this case, $\pa^*E^s$ has no vertical parts, but the length of its sections vanishes inside its projection, and rigidity fails.}}\label{fig basic}
  \end{figure}
see Figure \ref{fig basic}. In the seminal paper \cite{ChlebikCianchiFuscoAnnals05}, Chleb{\'{\i}}k, Cianchi, and Fusco provide a sufficient condition for the rigidity of equality cases in Steiner's inequality that is inspired by the above considerations. Precisely, they consider the localization of \eqref{steiner inequality} above a Borel set $\Om\subset\R^{n-1}$,
\begin{equation}\label{steiner inequality localizzata}
P(E;\Om\times\R)\ge P(F[v];\Om\times\R)\,,
\end{equation}
and show that, if (a) $\Om$ is an open connected set, (b) $v\in W^{1,1}(\Om)$, and (c) $v^\wedge>0$ $\H^{n-2}$-a.e. on $\Om$, then $E$ is $\H^n$-equivalent to a vertical translation of $F[v]$ inside $\Om\times\R$ whenever $P(E;\Om\times\R)=P(F[v];\Om\times\R)$. Assumption (b) and (c) express the requirements that $\pa^*F[v]$ has no vertical parts above $\Om$ and that the sections of $F[v]$ do not vanish inside the projection of $F[v]$. Although these conditions look reasonable in light of the examples depicted in Figure \ref{fig basic}, it is not hard to construct examples of sets in $\R^3$ such that rigidity holds true but either condition (b) or (c) fail. Moreover, as our analysis of Ehrhard's inequality suggests, the use of topological connectedness in assumption (a) should be unnecessary. By exploiting the ideas introduced in this paper, one can obtain several rigidity results for Steiner's inequality and largely extend the scope of previous rigidity theory. For example, in strikingly analogy with Theorem \ref{thm gauss}, one can show that if $v\in BV(\R^{n-1};[0,\infty))$ with $\H^{n-1}(\{v>0\})<\infty$ and $D^sv\llcorner\{v^\wedge>0\}=0$ (where $D^sv$ denotes the singular part of the distributional derivative of $v$), then, equivalently,
\begin{enumerate}
  \item[(i)] if $E\in\M(v)$, then $E$ is $\H^n$-equivalent to $t\,e_n+F[v]$ for some $t\in\R$;
  \item[(ii)] the set $\{v^\wedge=0\}$ does not essentially disconnect $\{v>0\}$;
  \item[(iii)] $F[v]$ is indecomposable.
\end{enumerate}
(Implication $(ii)\Rightarrow(i)$ of this theorem, coupled with an approximation argument, leads to a proof of the Chleb{\'{\i}}k-Cianchi-Fusco theorem.) Moreover, we can obtain suitable characterizations of rigidity even in the case when $\pa^*F[v]$ contains more substantial vertical parts than those allowed by the assumption $D^sv\llcorner\{v^\wedge>0\}=0$. We refer interested readers to the forthcoming paper \cite{ccdpmSTEINER} for a detailed account on these results.

\medskip

\noindent {\bf Acknowledgement}\,: This work was carried out while FC, MC, and GDP were visiting the University of Texas at Austin. The work of FC was partially supported by the UT Austin-Portugal partnership through the FCT post-doctoral fellowship SFRH/BPD/51349/2011. The work of GDP was partially supported by ERC under FP7, Advanced Grant n. 246923. The work of FM was partially supported by ERC under FP7, Starting Grant n. 258685 and Advanced Grant n. 226234, by the Institute for Computational Engineering and Sciences and by the Mathematics Department of the University of Texas at Austin during the time he was visiting these institutions, and by NSF Grant DMS-1265910.

\section{Notions from Geometric Measure Theory}\label{section preliminaries} We gather here some tools from Geometric Measure Theory. The notions needed in this paper are treated in adequate generality in the monographs \cite{GMSbook1,AFP,maggiBOOK}.

\subsection{General notation in $\R^n$}\label{section general notation} We denote by $B(x,r)$ and $\ov{B}(x,r)$ the open and closed Euclidean balls of radius $r>0$ and center $x\in\R^n$. Given $x\in\R^n$ and $\nu\in S^{n-1}$ we denote by $H_{x,\nu}^+$ and $H_{x,\nu}^-$ the complementary half-spaces
\begin{eqnarray}\label{Hxnu+}
  H_{x,\nu}^+&=&\Big\{y\in\R^n:(y-x)\cdot\nu\ge 0\Big\}\,,
  \\\nonumber
  H_{x,\nu}^-&=&\Big\{y\in\R^n:(y-x)\cdot\nu\le 0\Big\}\,.
\end{eqnarray}
Finally, we decompose $\R^n$ as the product $\R^{n-1}\times\R$, and denote by $\p:\R^n\to\R^{n-1}$ and $\q:\R^n\to\R$ the corresponding horizontal and vertical projections, so that $x=(\p x,\q x)=(x',x_n)$ and $x'=(x_1,\dots,x_{n-1})$ for every $x\in\R^n$. We set
\begin{eqnarray*}
\C_{x,r}&=&\Big\{y\in\R^n:|\p x- \p y|<r\,,|\q x- \q y|<r\Big\}\,,
\\
\D_{z,r}&=&\Big\{w\in\R^{n-1}:|w-z|<r\Big\}\,,
\end{eqnarray*}
for the vertical cylinder of center $x\in\R^n$ and radius $r>0$, and for the $(n-1)$-dimensional ball in $\R^{n-1}$ of center $z\in\R^{n-1}$ and radius $r>0$, respectively. In this way, $\C_{x,r}=\D_{\p x,r}\times(\q x-r,\q x+r)$. We shall use the following two notions of convergence for Lebesgue measurable subsets of $\R^n$. Given Lebesgue measurable sets $\{E_h\}_{h\in\N}$ and $E$ in $\R^n$, we shall say that $E_h$ locally converge to $E$, and write
\[
E_h\toloc E\,,\qquad\mbox{as $h\to\infty$}\,,
\]
provided $\H^n((E_h\Delta E)\cap K)\to 0$ as $h\to\infty$ for every compact set $K\subset\R^n$; we say that $E_h$ converge to $E$ as $h\to\infty$, and write $E_h\to E$, provided $\H^n(E_h\Delta E)\to 0$ as $h\to\infty$.

\subsection{Density points} If $E$ is a Lebesgue measurable set in $\R^n$ and $x\in\R^n$, then we define the upper and lower $n$-dimensional densities of $E$ at $x$ as
\begin{eqnarray*}
  \theta^*(E,x)=\limsup_{r\to 0^+}\frac{\H^n(E\cap \ov{B}(x,r))}{\om_n\,r^n}\,,
  \qquad
  \theta_*(E,x)=\liminf_{r\to 0^+}\frac{\H^n(E\cap \ov{B}(x,r))}{\om_n\,r^n}\,,
\end{eqnarray*}
respectively. In this way we define two Borel functions on $\R^n$, that agree a.e. on $\R^n$. In particular, the $n$-dimensional density of $E$ at $x$
\[
\theta(E,x)=\lim_{r\to 0^+}\frac{\H^n(E\cap \ov{B}(x,r))}{\om_n\,r^n}=\lim_{r\to 0}\frac{\H^n(E\cap B(x,r))}{\om_n\,r^n}\,,
\]
is defined for a.e. $x\in\R^n$, and $\theta(E,\cdot)$ is a Borel function on $\R^n$ (up to extending it by a constant value on some $\H^n$-negligible set). Correspondingly, for $t\in[0,1]$, we set $E^{(t)}=\{x\in\R^n:\theta(E,x)=t\}$. By the Lebesgue differentiation theorem, $\{E^{(0)},E^{(1)}\}$ is a partition of $\R^n$ up to a $\H^n$-negligible set. It is useful to keep in mind that
\begin{eqnarray*}
  &&x\in E^{(1)}\qquad\mbox{if and only if}\qquad E_{x,r}\toloc\R^n\quad\mbox{as $r\to 0^+$}\,,
  \\
  &&x\in E^{(0)}\qquad\mbox{if and only if}\qquad E_{x,r}\toloc\emptyset\quad\mbox{as $r\to 0^+$}\,,
\end{eqnarray*}
where $E_{x,r}$ denotes the blow-up of $E$ at $x$ at scale $r$, defined as
\begin{eqnarray*}
  E_{x,r}=\frac{E-x}{r}=\Big\{\frac{y-x}r:y\in E\Big\}\,,\qquad x\in\R^n\,,r>0\,.
\end{eqnarray*}
The set $\pae E=\R^n\setminus(E^{(0)}\cup E^{(1)})$ is called the essential boundary of $E$. Thus, in general, we only have $\H^n(\pae E)=0$, and we do not know $\pae E$ to be ``$(n-1)$-dimensional''.

\subsection{Approximate limits}\label{section approximate limits} Strictly related to the notion of density is that of approximate upper and lower limits of a measurable function. We shall stick to Federer's convention \cite[2.9.12]{FedererBOOK} in place of the one usually adopted in the study of functions of bounded variation \cite[Section 3.6]{AFP} since we will mainly deal with functions of {\it generalized} bounded variation; see section \ref{section sofp}. Given a Lebesgue measurable function $f:\R^n\to\R\cup\{\pm\infty\}$ we define the (weak) approximate upper and lower limits of $f$ at $x\in\R^n$ as
  \begin{eqnarray*}
  f^\vee(x)&=&\inf\Big\{t\in\R:\theta(\{f>t\},x)=0\Big\}=\inf\Big\{t\in\R:\theta(\{f<t\},x)=1\Big\}\,,
  \\
  f^\wedge(x)&=&\sup\Big\{t\in\R:\theta(\{f<t\},x)=0\Big\}=\sup\Big\{t\in\R:\theta(\{f>t\},x)=1\Big\}\,.
  \end{eqnarray*}
 Note that $f^\vee$ and $f^\wedge$ are Borel functions with values on $\R\cup\{\pm\infty\}$ defined \textit{at every point $x$} of $\R^n$, and they do not depend on the representative chosen for the function $f$.
 The approximate jump of $f$ is the Borel function $[f]:\R^n\to[0,\infty]$ defined by
 \[
 [f](x)=f^\vee(x)-f^\wedge(x)\,,\qquad x\in\R^n\,.
 \]
 We easily deduce the the following properties, which hold true for every Lebesgue measurable $f:\R^n\to\R\cup\{\pm\infty\}$ and for every $t \in \mathbb{R}$:
\begin{align}
\{ |f|^{\vee}  < t \} &=\{ - t < f^{\wedge} \} \cap \{ f^{\vee} < t \}\,, \label{dens1}
\\
\{ f^{\vee} < t \} &\subset \{ f < t \}^{(1)}\subset\{f^\vee\le t\}\,, \label{dens2}
\\
\{ f^{\wedge} > t \} &\subset \{ f > t \}^{(1)}\subset\{f^\wedge \ge t\}\,. \label{dens3}
\end{align}
(Note that all the inclusions may be strict, that we also have $\{ f < t \}^{(1)}=\{ f^\vee < t \}^{(1)}$, and that all the other analogous relations hold true.) If $f$ is non negative and $E$ is Lebesgue measurable, then for every $x \in E^{(1)}$, we have
\begin{equation}\label{gino1}
( 1_{E}  f)^{\vee} (x) = f^{\vee} (x) \,,
\qquad ( 1_{E}  f)^{\wedge} (x) = f^{\wedge} (x)\,\,.
\end{equation}
Finally, we notice that if $I$ and $J$ are intervals in $\R\cup\{\pm\infty\}$, $\vphi:I\to J$ is continuous and decreasing, and $f$ takes values into $I$, then $v=\vphi\circ f$ is Lebesgue measurable on $\R^n$, with
\begin{equation}
  \label{continuous and decreasing}
  v^\wedge=\vphi(f^\vee)\,,\qquad v^\vee=\vphi(f^\wedge)\,.
\end{equation}
We now introduce the set of approximate discontinuity points $S_f$ of a Lebesgue measurable function $f:\R^n\to\R\cup\{\pm\infty\}$, which is defined as
\[
S_f=\Big\{x\in\R^n:f^\wedge(x)<f^\vee(x)\Big\}=\Big\{x\in\R^n:[f](x)>0\Big\}\,.
\]
We have the following general fact, that is usually stated in the finite-valued case only. For this reason we have included the short proof.

\begin{proposition}
  If $f:\R^n\to\R\cup\{\pm\infty\}$ is Lebesgue measurable, then $\{f^\wedge=f^\vee=f\}$ is $\H^n$-equivalent to $\R^n$. In particular, $f^\vee$ and $f^\wedge$ are representatives of $f$, and $\H^n(S_f)=0$.
\end{proposition}

\begin{proof}
  Let us consider the function $\Phi$ defined in \eqref{definition of Phi}. Since $\Phi:\R\cup\{\pm \infty\}\to[0,1]$ is continuous and decreasing, it turns out that $v=\Phi\circ f:\R^n\to[0,1]$ is Lebesgue measurable, with $v^\vee=\Phi\circ f^\wedge$ and $v^\wedge=\Phi\circ f^\vee$. Thus $S_v=S_f$, where, by \cite[Section 3.1.4, Proposition 3]{GMSbook1}, $\H^n(S_v)=0$.
\end{proof}

If $f:\R^n\to\R\cup\{\pm\infty\}$ and $A\subset\R^n$ Lebesgue measurable, then we say that $t\in\R\cup\{\pm\infty\}$ is the approximate limit of $f$ at $x$ with respect to $A$, and write $t=\aplim (f,A,x)$, if
\begin{eqnarray*}
  &&\theta\Big(\{|f-t|>\e\}\cap A;x\Big)=0\,,\qquad\forall\e>0\,,\hspace{0.3cm}\qquad (t\in\R)\,,
  \\
  &&\theta\Big(\{f<M\}\cap A;x\Big)=0\,,\qquad\hspace{0.6cm}\forall M>0\,,\qquad (t=+\infty)\,,
  \\
  &&\theta\Big(\{f>-M\}\cap A;x\Big)=0\,,\qquad\hspace{0.3cm}\forall M>0\,,\qquad (t=-\infty)\,.
\end{eqnarray*}
We say that $x\in S_f$ is a jump point of $f$ if there exists $\nu\in S^{n-1}$ such that
\[
f^\vee(x)=\aplim(f,H_{x,\nu}^+,x)\,,\qquad f^\wedge(x)=\aplim(f,H_{x,\nu}^-,x)\,.
\]
If this is the case we set $\nu=\nu_f(x)$, the approximate jump direction of $f$ at $x$. We denote by $J_f$ the set of approximate jump points of $f$, so that $J_f\subset S_f$; moreover, $\nu_f:J_f\to S^{n-1}$ is a Borel function. It will be particularly useful to keep in mind the following proposition.

\begin{proposition}\label{lemma nonciserve}
  We have that $x\in J_f$ if and only if for every $\e>0$ such that $f^\wedge(x)+\e<f^\vee(x)-\e$ we have
  \begin{eqnarray*}
    \{|f-f^\vee(x)|\le\e\}_{x,r}\toloc H_{0,\nu}^+\,,\qquad\{|f-f^\wedge(x)|\le\e\}_{x,r}\toloc H_{0,\nu}^-\,,\qquad\mbox{as $r\to 0^+$}\,.
  \end{eqnarray*}
  Similarly, $x\in J_f$ if and only if for every $\tau\in(f^\wedge(x),f^\vee(x))$ we have
  \begin{eqnarray}\label{jump point level sets}
    \{f>\tau\}_{x,r}\toloc H_{0,\nu}^+\,,\qquad\{f<\tau\}_{x,r}\toloc H_{0,\nu}^-\,,\qquad\mbox{as $r\to 0^+$}\,.
  \end{eqnarray}
\end{proposition}

\begin{proof} We prove the ``only if'' part of the first equivalence only, leaving the other implications to the reader. Let us set $t=f^\vee(x)$ and $s=f^\wedge(x)$. By assumption
  \begin{eqnarray*}
   \Big(\Big\{|f-t|>\e\Big\}\cap H_{x,\nu}^+\Big)_{x,r}\toloc\emptyset\,,
   \qquad
   \Big(\Big\{|f-s|>\e\Big\}\cap H_{x,\nu}^-\Big)_{x,r}\toloc\emptyset\,,
  \end{eqnarray*}
  as $r\to 0^+$. As a consequence, as $r\to 0^+$,
  \begin{eqnarray*}
   \Big(\Big\{|f-t|\le\e\Big\}\cup H_{x,\nu}^-\Big)_{x,r}\toloc\R^n\,,
   \qquad
   \Big(\Big\{|f-s|\le\e\Big\}\cup H_{x,\nu}^+\Big)_{x,r}\toloc\R^n\,.
  \end{eqnarray*}
  As $E^{(1)}\cap F^{(1)}=(E\cap F)^{(1)}$, we find
   \begin{eqnarray*}
   &&\Big(\Big(\Big\{|f-t|\le\e\Big\}\cup H_{x,\nu}^-\Big)
   \cap\Big(\Big\{|f-s|\le\e\Big\}\cup H_{x,\nu}^+\Big)\Big)_{x,r}\toloc\R^n\,,
   \\\
   \mbox{that is}\hspace{0.5cm}&&\Big(\Big\{|f-t|\le\e\Big\}\cap H_{x,\nu}^+\Big)_{x,r}
   \cup\Big(\Big\{|f-s|\le\e\Big\}\cap H_{x,\nu}^-\Big)_{x,r}\toloc\R^n\,,
  \end{eqnarray*}
  Since the two sets are disjoint, the first one contained in $H_{0,\nu}^+$, the second one in $H_{0,\nu}^-$, we complete the proof.
\end{proof}


\subsection{Rectifiable sets}\label{section preliminaries rect} Let $1\le k\le n$, $k\in\N$. A Borel set $M\subset\R^n$ is {\it countably $\H^k$-rectifiable} if there exist  Lipschitz functions $f_h:\R^k\to\R^n$ ($h\in\N$) such that
\begin{equation}
  \label{countably rectifiable}
  \H^k\bigg(M\setminus\bigcup_{h\in\N}f_h(\R^k)\bigg)=0\,.
\end{equation}
We further say that $M$ is {\it locally $\H^k$-rectifiable} if $\H^k(M\cap K)<\infty$ for every compact set $K\subset\R^n$, or, equivalently, if $\H^k\llcorner M$ is a Radon measure on $\R^n$. Hence, for a locally $\H^k$-rectifiable set $M$ in $\R^n$ the following definition is well-posed: we say that $M$ has a $k$-dimensional subspace $L$ of $\R^n$ as its approximate tangent plane at $x\in\R^n$, $L=T_xM$, if
\[
\lim_{r\to 0^+}\frac1{r^k}\int_{B(x,r)\cap M}\vphi\Big(\frac{y-x}r\Big)\,d\H^k(y)=\int_{L}\,\vphi\,d\H^k\,,\qquad\forall\vphi\in C^0_c(\R^n)\,.
\]
It turns out that $T_xM$ exists and is uniquely defined at $\H^k$-a.e. $x\in M$. Moreover, given two locally $\H^k$-rectifiable sets $M_1$ and $M_2$ in $\R^n$, it turns out that $T_xM_1=T_xM_2$ for $\H^k$-a.e. $x\in M_1\cap M_2$. Since $f(\R^k)$ is locally $\H^k$-rectifiable whenever $f:\R^k\to\R^n$ is a Lipschitz function, if $M$ is merely a countably $\H^k$-rectifiable set and $\{f_h\}_{h\in\N}$ is a sequence of Lipschitz functions  satisfying \eqref{countably rectifiable}, then we can find a partition modulo $\H^k$ of $M$ into Borel sets $\{M_h\}_{h\in\N}$ such that $T_xf(\R^k)$ exists for every $x\in M_h$: correspondingly, we set $T_xM=T_xf_h(\R^k)$ for $x\in M_h$. The definition is well-posed in the sense that the approximate tangent spaces defined by another family of Lipschitz functions $\{g_h\}_{h\in\N}$ satisfying \eqref{countably rectifiable} will just coincide at $\H^k$-a.e. $x\in M$ with the ones defined by $\{f_h\}_{h\in\N}$. In other words, $\{T_xM\}_{x\in M}$ is well-defined as an equivalence class modulo $\H^k$ of Borel functions from $M$ to the set of $k$-planes in $\R^n$.

Finally, we mention the following consequence of \cite[3.2.23]{FedererBOOK}: if $M$ is countably $\H^k$-rectifiable in $\R^n$, then $M\times\R^\ell$ is countably $\H^{k+\ell}$-rectifiable in $\R^{n+\ell}$, and
\begin{equation}
  \label{federer 3.2.23}
  (\H^k\llcorner M)\times\H^{\ell}=\H^{k+\ell}\llcorner \Big(M\times\R^\ell\Big)\,.
\end{equation}

\subsection{Functions of bounded variation and sets of finite perimeter}\label{section sofp} Given an open set $\Om\subset\R^n$ and $f\in L^1(\Om)$, we say that $f$ has bounded variation in $\Om$, $f\in BV(\Om)$, if the total variation of $f$ in $\Om$, defined as
\[
|Df|(\Om)=\sup\Big\{\int_\Om\,f(x)\,\Div\,T(x)\,dx:T\in C^1_c(\Om;\R^n)\,,|T|\le 1\Big\}\,,
\]
is finite. We say that $f\in BV_{loc}(\Om)$ if $f:\Om\to\R$ is Lebesgue measurable, and, for every open set $\Om'\cc\Om$, we have $f\in BV(\Om')$. If $f\in BV_{loc}(\R^n)$ then the distributional derivative $Df$ of $f$ is an $\R^n$-valued Radon measure. The Radon--Nykodim decomposition of $Df$ with respect to $\H^n$ is denoted by $Df=D^af+D^sf$, where $D^sf$ and $\H^n$ are mutually singular, and where $D^af\ll\H^n$.
Moreover, $S_f$ is countably $\H^{n-1}$-rectifiable, with $\H^{n-1}(S_f\setminus J_f)=0$, $[f]\in L^1_{loc}(\H^{n-1}\llcorner J_f)$, and the $\R^n$-valued Radon measure $D^jf$, defined as $D^jf=[f]\,\nu_f\,d\H^{n-1}\llcorner J_f$, is called the jump part of $Df$. Since $D^af$ and $D^jf$ are mutually singular, by setting $D^cf=D^sf-D^jf$ we come to the canonical decomposition of $Df$ into the sum $D^af+D^jf+D^cf$, where $D^cf$ is called the Cantorian part of $Df$. It turns out that $|D^cf|(M)=0$ whenever $M$ is $\s$-finite with respect to $\H^{n-1}$.

A Lebesgue measurable set $E\subset\R^n$ is said of locally finite perimeter in $\R^n$ if $1_E\in BV_{loc}(\R^{n})$. In this case, we call $\mu_E=-D1_E$ the Gauss--Green measure of $E$, so that
\[
\int_E\nabla\vphi(x)\,dx=\int_{\R^n}\vphi(x)\,d\mu_E(x)\,,\qquad\forall \vphi\in C^1_c(\R^n)\,.
\]
The reduced boundary of $E$ is the set $\pa^*E$ of those $x\in\R^n$ such that
\[
\nu_E(x)=\lim_{r\to 0^+}\,\frac{\mu_E(B(x,r))}{|\mu_E|(B(x,r))}\qquad\mbox{exists and belongs to $S^{n-1}$}\,.
\]
The Borel function $\nu_E:\pa^*E\to S^{n-1}$ is called the measure-theoretic outer unit normal to $E$. It turns out that $\pa^*E$ is a locally $\H^{n-1}$-rectifiable set in $\R^n$ \cite[Corollary 16.1]{maggiBOOK}, that $\mu_E=\nu_E\,\H^{n-1}\llcorner\pa^*E$, so that
\[
\int_E\nabla\vphi(x)\,dx=\int_{\pa^*E}\vphi(x)\,\nu_E(x)\,d\H^{n-1}(x)\,,\qquad\forall \vphi\in C^1_c(\R^n)\,.
\]
We say that $x\in\R^n$ is a jump point of $E$, if and only if there exists $\nu\in S^{n-1}$ such that
\begin{equation}
  \label{jump point of E}
  E_{x,r}\toloc H_{0,\nu}^+\,,\qquad\mbox{as $r\to 0^+$}\,,
\end{equation}
and we denote by $\pa^JE$ the set of jump points of $E$. Notice that we always have $\pa^JE\subset E^{(1/2)}\subset\pae E$. In fact, if $E$ is a set of locally finite perimeter and $x\in\pa^*E$, then \eqref{jump point of E} holds true with $\nu=-\nu_E(x)$, so that $\pa^*E\subset\pa^JE$. Summarizing, if $E$ is a set of locally finite perimeter, we have
\begin{equation}
  \label{inclusioni frontiere}
  \pa^*E\subset\pa^J E\subset E^{1/2}\subset \pae E\,,
\end{equation}
and, moreover, by {\it Federer's theorem} \cite[Theorem 3.61]{AFP}, \cite[Theorem 16.2]{maggiBOOK},
\[
\H^{n-1}(\pae E\setminus\pa^*E)=0\,,
\]
so that $\pae E$ is locally $\H^{n-1}$-rectifiable in $\R^n$. We shall also need the following criterion for finite perimeter, known as {\it Federer's criterion} \cite[4.5.11]{FedererBOOK} (see also \cite[Theorem 1, section 5.11]{EvansGariepyBOOK}): if $E$ is a Lebesgue measurable set in $\R^n$ such that
\[
\H^{n-1}(K\cap\pae E)<\infty\,,\qquad\mbox{for every compact set $K\subset\R^n$}\,,
\]
then $E$ is a set of locally finite perimeter. (Notice that Federer's criterion is actually more general than this.) We conclude this preliminary section by the following remark, which shows the equivalence for a set of locally finite perimeter between being indecomposable and being essentially connected (see section \ref{section connectedness} for the terminology).

\begin{remark}\label{remark indecomposable 2}{\rm
  If $E$ is an indecomposable set in $\R^n$, then, whenever $\{F,G\}$ is a non-trivial partition of $E$ {\it by Lebesgue measurable sets}, we have
  \begin{equation}
    \label{ambrosio}
      \H^{n-1}\Big(E^{(1)}\cap\pae F\cap\pae G\Big)>0\,.
  \end{equation}
  Indeed, in the case that $\{F,G\}$ is further assumed to be a partition by sets of locally finite perimeter, then, by definition of indecomposability, there exists $r_0$ such that $P(E;B_r)<P(F;B_r)+P(G;B_r)$ for every $r>r_0$. Thus, by Federer's theorem,
  \begin{eqnarray}\nonumber
  \H^{n-1}(B_r\cap\pae E)&<&\H^{n-1}(B_r\cap\pae F)+\H^{n-1}(B_r\cap\pae G)
  \\\label{burger}
  &=&\H^{n-1}(B_r\cap\pae F\cap\pae E)+\H^{n-1}(B_r\cap\pae G\cap\pae E)\hspace{0.3cm}
  \\\nonumber
  &&+\H^{n-1}(B_r\cap\pae F\cap E^{(1)})+\H^{n-1}(B_r\cap\pae G\cap E^{(1)})
  \end{eqnarray}
  where we have used the fact that, since $F\subset E$, then $\pae F=(\pae F\cap \pae E)\cup(\pae F\cap E^{(1)})$ (a similar remark is applied to $G$ too). Since $(\pae F\Delta\pae G)\cap(E^{(1)}\cup E^{(0)})=\emptyset$ and $\pa^JF\cap\pa^JG\subset E^{(1)}$, by Federer's theorem we find that
  $\pae F\Delta\pae G$ is $\H^{n-1}$-equivalent to $\pae E$. Hence, \eqref{burger} is equivalent to $0<2\,\H^{n-1}(\pae F\cap\pae G\cap E^{(1)}\cap B_r)$ for every $r>r_0$, that is, \eqref{ambrosio}. To settle the general case, let us assume, arguing by contradiction, the existence of a non-trivial Lebesgue measurable partition $\{F,G\}$ of $E$ such that
  \begin{equation}
    \label{no ambrosio}
   0=\H^{n-1}\Big(E^{(1)}\cap\pae F\cap\pae G\Big)=\H^{n-1}\Big((\pae F\cap\pae G)\setminus\pae E\Big)\,.
  \end{equation}
  We are now going to show that, in this case, $F$ and $G$ are necessarily sets of locally finite perimeter, thus contradicting the fact that $E$ is  indecomposable. Indeed, since $F\subset E$, we have $E^{(0)}\subset F^{(0)}$, and thus $\pae F\cap E^{(0)}=E^{(0)}\setminus(F^{(0)}\cup F^{(1)})=\emptyset$; thus
  \begin{equation}
    \label{no ambrosio2}
  \pae F\subset\pae E\cup (\pae F\cap E^{(1)})\,.
  \end{equation}
  At the same time, since $\pae  F\cap E^{(1)}\subset \pae F\cap\pae G$, we find
  \[
  \pae  F\cap E^{(1)}\subset \Big(\pae F\cap\pae G\Big)\setminus\pae E\,.
  \]
  Therefore, by \eqref{no ambrosio} and \eqref{no ambrosio2}, for every compact set $K\subset\R^n$, and since $E$ is of locally finite perimeter,
  $\H^{n-1}(K\cap\pae F)\le  \H^{n-1}(K\cap\pae E)<\infty$. By Federer's criterion, $F$ is a set of locally finite perimeter, and so is $G=E\setminus F$. We can thus repeat our initial argument to prove that $\H^{n-1}(E^{(1)}\cap\pae F\cap\pae G)>0$ and obtain a contradiction.}
\end{remark}

\section{Rigidity of equality cases in Ehrhard inequality}\label{section gauss} This section contains the proofs of Theorem \ref{thm gauss} and Theorem \ref{thm gauss R2}. In section \ref{subsection epigraphs} we collect the basic results concerning epigraphs of locally finite perimeter. In section \ref{section gauss ii implica i} we show the implication $(ii)\Rightarrow(i)$ in Theorem \ref{thm gauss}, while in section \ref{section gauss i implica ii} we prove the implication $(i)\Rightarrow(ii)$. In section \ref{section gauss conclusion}, we finally prove Theorem \ref{thm gauss R2}.

\subsection{Epigraphs of locally finite perimeter and the space $GBV_*$}\label{subsection epigraphs} Let us set
\[
\S_f=\{x\in\R^n:\q x>f(\p x)\}\,.
\]
for the epigraph of $f:\R^{n-1}\to\R\cup\{\pm\infty\}$. In this section we analyze the situation when $f$ defines an epigraph of locally finite perimeter. To this end, it is convenient to introduce the functions $\tau_M:\R\to\R$ ($M>0$) defined as
\[
\tau_M(s)=\max\Big\{-M,\min\Big\{M,s\Big\}\Big\}\,,\qquad s\in\R\cup\{\pm\infty\}\,,
\]
and set the following definition: a Lebesgue measurable function $f:\R^{n-1}\to\R\cup\{\pm\infty\}$ is a function of generalized bounded variation with values in extended real numbers, $f\in GBV_*(\R^{n-1})$, if $\tau_M(f)\in BV_{loc}(\R^{n-1})$ for every $M>0$, or, equivalently, if $\psi(f)\in BV_{loc}(\R^{n-1})$ for every $\psi\in C^1(\R)$ with $\psi'\in C^0_c(\R)$. (Note that the composition makes sense since, for example, there will be positive constants $c$ and $t_0$ such that $\psi(t)=c$ for every $t>t_0$: correspondingly, we shall set $\psi(f)=c$ on $\{f=\infty\}$, and argue similarly on the set $\{f=-\infty\}$.) If we start from Lebesgue measurable functions $f:\R^{n-1}\to\R$, we shall set $GBV(\R^{n-1})$ for the corresponding space. The space $GBV_*(\R^{n-1})$ plays a particularly important role in our analysis because of the following proposition.

\begin{proposition}\label{prop f GBV}
  If $f:\R^{n-1}\to\R\cup\{\pm\infty\}$ is Lebesgue measurable, then $f\in GBV_*(\R^{n-1})$ if and only if $\S_f$ is of locally finite perimeter in $\R^n$; moreover, in this case, for a.e. $t\in\R$, we have that $\{f<t\}$ is a set of locally finite perimeter in $\R^{n-1}$.
\end{proposition}

\begin{remark}
  {\rm If $\Om\subset\R^{n-1}$ is an open set and $f\in L^1(\Om)$, it is well-known that $f\in BV(\Om)$ if and only if $\S_f$ is of finite perimeter in $\Om\times\R$; see, e.g. \cite[Section 4.1.5]{GMSbook1}. This result, because of the artificial structures assumed in it (open set and summable function) will not suffice for our purposes. Moreover, it seems that the infinite-valued case is not covered by the literature. Therefore, we shall provide a proof of Proposition \ref{prop f GBV}. Similar remarks apply to Proposition \ref{proposition fine bv} and Lemma \ref{lemma fico} below. We also notice that we shall need to refer to these proofs in some crucial steps of the proof of Theorem \ref{thm gauss}.}
\end{remark}

\begin{remark}\label{remark f in GBV}
  {\rm Note that if $v\in BV_{loc}(\R^{n-1};[0,1])$, then $f=\Psi\circ v\in GBV_*(\R^{n-1})$, where $\Psi$ is defined as in \eqref{definition of Phi}. Indeed, if we pick any $\psi\in C^1(\R)$ with $\psi'\in C^0_c(\R)$, then $\psi\circ\Psi$ is real-valued on $[0,1]$, with $\psi\circ\Psi\in C^1([0,1])$ and $(\psi\circ\Psi)'\in C^0_c((0,1))$. Therefore, $\psi\circ f=(\psi\circ\Psi)\circ v\in BV_{loc}(\R^{n-1})$ by the $C^1$ chain rule theorem on $BV$.}
\end{remark}

\begin{proof}[Proof of Proposition \ref{prop f GBV}]
  {\it Step one}\,: We show that if $\S_f$ is of locally finite perimeter then $f\in GBV_*(\R^{n-1})$. Let $\psi\in C^1(\R)$ with $\psi'\in C^0_c(\R)$, so that $\psi\circ f$ is defined on $\R^{n-1}$ with $\psi\circ f\in L^\infty(\R^{n-1})\subset L^1_{loc}(\R^{n-1})$. If $\psi\in C^2(\R)$, then  $\psi'(\q x)\vphi(\p x)\in C^1_c(\R^n)$ for every $\vphi\in C^1_c(\R^{n-1})$, and thus, setting $\nabla'=(\pa_1,\dots,\pa_{n-1})$,
  \begin{eqnarray*}
    \Big|\int_{\S_f}\nabla'(\psi'(\q x)\vphi(\p x))\,dx\bigg|&=&\Big|\int_{\pa^*{\S_f}}\psi'(\q x)\,\vphi(\p x)\,\p\nu_{\S_f}(x)\,d\H^{n-1}(x)\Big|
    \\
    &\le&\Lip(\psi)\,\sup|\vphi|\, P({\S_f};\spt\vphi\times\spt\psi')\,.
  \end{eqnarray*}
  At the same time, by Fubini's theorem
  \[
  \int_{\S_f}\nabla'(\psi'(\q x)\vphi(\p x))\,dx=\int_{\R^{n-1}}\nabla'\vphi(z)\,dz\int_{f(z)}^\infty\,\psi'(t)\,dt
  = - \int_{\R^{n-1}}\psi(f(z))\,\nabla'\vphi(z)\,dz\,.
  \]
  Hence, for every $R>0$,
  \begin{eqnarray*}
  \sup\Big\{\Big|\int_{\R^{n-1}}(\psi\circ f)\,\nabla'\vphi \Big|:\vphi\in C^1_c(\D_R)\,,|\vphi|\le1\Big\}
  \le\Lip(\psi)\, P({\S_f};\D_R\times\spt\psi')<\infty\,,
  \end{eqnarray*}
  that is, $\psi(f)\in BV_{loc}(\R^{n-1})$ if $\psi\in C^2(\R)$. By approximation, the same holds if we only have $\psi\in C^1(\R)$, and thus, $f\in GBV_*(\R^{n-1})$.

  \medskip

  \noindent {\it Step two}\,: If $f\in GBV_*(\R^{n-1})$, then $\tau_M\circ f\in BV_{loc}(\R^{n-1})$, $\{\tau_M\circ f<t\}=\{f<t\}$ for every $|t|<M$, and $\{\tau_M\circ f<t\}$ is of locally finite perimeter for a.e. $t\in\R$. Hence, $\{f<t\}$ is of locally finite perimeter for a.e. $t\in\R$. Let now $\vphi\in C^1_c(\R^n)$, with $\spt\,\vphi\cc \D_R\times(-R,R)$ for some $R>0$. On the one hand, we have
  \begin{eqnarray}\label{pino1}
    \Big|\int_{\S_f}\pa_n\vphi\Big|=\Big|\int_{\R^{n-1}}\,dz\int_{f(z)}^\infty\pa_n\vphi\Big|\le \sup_{\R^n}|\vphi|\,\H^{n-1}(\D_R)\,;
  \end{eqnarray}
  on the other hand, since $\{f<t\}$ is of locally finite perimeter for a.e. $t\in\R$, we find
  \begin{eqnarray}\nonumber
    \Big|\int_{\S_f}\nabla'\vphi\Big|&=& \Big|\int_\R dt\int_{\{f<t\}}\nabla'\vphi(z,t)\,dz\Big|
    = \Big|\int_\R dt\int_{\pa^*\{f<t\}}\vphi(z,t)\,\nu_{\{f<t\}}(z)\,d\H^{n-2}(z)\Big|
    \\\label{pino2}
    &\le& \sup_{\R^{n}}|\vphi|\,\int_{-R}^{R} P(\{f<t\};\D_R)\,dt=\sup_{\R^{n}}|\vphi|\,|D(\tau_R\circ f)|(\D_R)\,,
  \end{eqnarray}
  by coarea formula. By \eqref{pino1} and \eqref{pino2}, $\S_f$ is a set of locally finite perimeter.
\end{proof}

Given a Lebesgue measurable function $f:\R^{n-1}\to\R\cup\{\pm\infty\}$, we set
\begin{eqnarray*}
\G_f&=&\Big\{x\in\R^n :f^\wedge(\p x)\le \q x\le f^\vee(\p x)\Big\}\,,
\\
\Gv_f&=&\Big\{x\in\R^n :f^\wedge(\p x)< \q x< f^\vee(\p x)\Big\}\,.
\end{eqnarray*}
We call $\G_f$ the complete graph of $f$, and $\Gv_f$ the vertical graph of $f$. Note that these objects are invariant in the $\H^{n-1}$-equivalence class of $f$.

\begin{proposition}\label{proposition fine bv}
  If $f\in GBV_*(\R^{n-1})$, then
  \begin{eqnarray}\label{bonn1}
  &&\pa^*\S_f\cap(S_f^c\times\R)=_{\H^{n-1}}\,\Big\{x\in\R^n:\q x=f^\wedge(\p x)=f^\vee(\p x)\Big\}\,,
  \\\label{bonn2}
  &&\pa^*\S_f\cap(S_f\times\R)=_{\H^{n-1}}\,\Gv_f\,,
  \\
  \label{Sf1}
  &&\hspace{2.1cm}\S_f^{(1)}=_{\H^{n-1}}\Big\{x\in\R^n:\q x>f^\vee(\p x)\Big\}\,,
  \\\label{Sf0}
  &&\hspace{2.1cm}\S_f^{(0)}=_{\H^{n-1}}\Big\{x\in\R^n:\q x<f^\wedge(\p x)\Big\}\,.
  \end{eqnarray}
  Moreover, $S_f$ is countably $\H^{n-2}$-rectifiable with $\H^{n-2}(S_f\setminus J_f)=0$. Finally, for $\H^{n-1}$-a.e. $x\in\Gv_f$, the outer unit normal $\nu_{\S_f}(x)$ exists, $S_f$ has an approximate tangent plane at $\p x$, and
  $\nu_{\S_f}(x)=(\nu_{S_f}(\p x),0)$, where $\nu_{S_f}(\p x)$ is a unit normal direction to $T_{\p x}S_f$ in $\R^{n-1}$.
\end{proposition}

\begin{remark}
  {\rm Here and in the following, $A=_{\H^k}B$ stands for $\H^k(A\Delta B)=0$.}
\end{remark}

Proposition \ref{proposition fine bv} is in turn based on the following lemma, that will play a crucial role also in the proof of Theorem \ref{thm gauss}.

\begin{lemma}
  \label{lemma fico}
  If $f:\R^{n-1}\to\R\cup\{\pm\infty\}$ is a Lebesgue measurable function, $I$ is a countable dense subset of $\R$ with the property that $\{f>t\}$ is of locally finite perimeter for every $t\in I$, and if we set
  \[
  N_f=\bigcup_{t\in I}\pae \{f>t\}\setminus\pa^*\{f>t\}\,,
  \]
  then $\H^{n-2}(N_f)=0$, and for every $z\in S_f\setminus N_f$ there exists $\nu(z)\in S^{n-2}$ such that
  \[
  z\in\pa^J\{f>t\}\,,\qquad\forall t\in(f^\wedge(z),f^\vee(z))\,,
  \]
  with jump direction $\nu(z)$. (In other words, the jump direction of $\{f>t\}$ at $z$ is independent of $t$). In particular, we have
  \[
  S_f\setminus N_f\subset J_f\,,\qquad \H^{n-2}(S_f\setminus J_f)=0\,.
  \]
\end{lemma}

\begin{remark}
  {\rm Notice that the set $N_f$ depends also on the choice of $I$.}
\end{remark}

\begin{proof}
  [Proof of Lemma \ref{lemma fico}] By Federer's theorem, $\H^{n-2}(N_f)=0$. We now notice that,
\[
\left\{
\begin{array}{l}
  z\in S_f\,,
  \\
  f^\wedge(z)<t<s<f^\vee(z)\,,
\end{array}
\right .
\qquad\Rightarrow\qquad z\in \pae \{f>t\}\cap\pae \{f>s\}\,.
\]
By taking into account that $z\in S_f\setminus N_f$ if and only if $z\in S_f$ and for every $t\in I$ either $z\not\in\pae \{f>t\}$ or $z\in\pa^*\{f>t\}$, we thus find
\begin{eqnarray}\nonumber
\left\{
\begin{array}{l}
  z\in S_f\setminus N_f\,,
  \\
  f^\wedge(z)<t<s<f^\vee(z)\,,
  \\
  t,s\in I\,,
\end{array}
\right .
&&\quad\Rightarrow\qquad z\in \pa^*\{f>t\}\cap\pa^*\{f>s\}
\\\nonumber
&&\quad\Rightarrow\qquad \{f>t\}_{z,r}\toloc H_{0,\nu(z)}^+\,,\quad \{f>s\}_{z,r}\toloc H_{0,\nu(z)}^+\,,
\\\label{c}
&&\quad\qquad\qquad\mbox{where $-\nu(z)=\nu_{\{f>t\}}(z)=\nu_{\{f>s\}}(z)$\,,}
\end{eqnarray}
as $E\subset F$ implies indeed that $\nu_E=\nu_F$ on $\pa^*E\cap\pa^*F$. In other words, for every $z\in S_f\setminus N_f$ there exists $\nu(z)\in S^{n-2}$ such that
\[
\{f>t\}_{z,r}\toloc H_{0,\nu(z)}^+\,,\qquad \forall t\in I\cap(f^\wedge(z),f^\vee(z))\,.
\]
Finally, if $z\in S_f\setminus N_f$ with $f^\wedge(z)<t<f^\vee(z)$, then we may pick $s,s'\in I$ with $f^\wedge(z)<s<t<s'<f^\vee(z)$ and use
\[
\{f>s\}_{z,r}\toloc H_{0,\nu(z)}^+\,,\qquad \{f>s'\}_{z,r}\toloc H_{0,\nu(z)}^+\,,
\]
to infer $\{f>t\}_{z,r}\toloc H_{0,\nu(z)}^+$. Indeed, as a general fact, if $E_h\subset F_h\subset G_h$ with $E_h\to E$ and $G_h\to E$ as $h\to\infty$, then $F_h\to E$ as $h\to\infty$.
\end{proof}

\begin{proof}[Proof of Proposition \ref{proposition fine bv}] {\it Step one}\,: We show that $S_f$ is countably $\H^{n-2}$-rectifiable. Let $I\subset\R$ be a countable dense set in $\R$ such that for every $t\in I$ we have $\{f>t\}$ of locally finite perimeter in $\R^{n-1}$. By Federer's theorem, if $t\in I$, then $\pa^*\{f>t\}$ is locally $\H^{n-2}$-rectifiable, with $\H^{n-2}(\pae \{f>t\}\setminus\pa^*\{f>t\})=0$. Since $t<f^\vee(z)$ gives $\theta^*(\{f>t\},z)>0$, while $t>f^\wedge(z)$ implies $\theta_*(\{f>t\},z)<1$, we find that for every $t\in\R$
\[
\Big\{z\in\R^{n-1}:f^\vee(z)>t>f^\wedge(z)\Big\}\subset\pae \{f>t\}\,,
\]
so that, as $I$ is dense in $\R$,
\[
S_f\subset \bigcup_{t\in I}\Big\{z\in\R^{n-1}:f^\vee(z)>t>f^\wedge(z)\Big\}\subset\bigcup_{t\in I}\pae \{f>t\}\,.
\]
Thus $S_f$ is countably $\H^{n-2}$-rectifiable, as, by Federer's theorem and since $I$ is countable,
\[
\H^{n-2}\Big(S_f\setminus \bigcup_{t\in I}\pa^*\{f>t\}\Big)=0\,.
\]

\medskip

\noindent {\it Step two}\,: We prove that 
\begin{eqnarray}\label{a}
\pae\S_f\cap(S_f\times\R)&\subset_{\H^{n-1}}&\Gv_f\,,
\\\label{b}
\pae\S_f\cap(S_f^c\times\R)&\subset&\Big\{x\in\R^n:\q x=f^\wedge(\p x)=f^\vee(\p x)\Big\}\,,
\\\label{brigthon a}
\{x\in\R^n:\q x<f^\wedge(\p x)\}&\subset&\S_f^{(0)}\,,
\\\label{brigthon b}
\{x\in\R^n:\q x>f^\vee(\p x)\}&\subset&\S_f^{(1)}\,.
\end{eqnarray}
We start proving \eqref{brigthon a}: if $x\in\R^n$ is such that $\q x<f^\wedge(\p x)$, then $f^\wedge(\p x)>-\infty$ and, taking $t^*>\q x$ with $\theta(\{f<t^*\},\p x)=0$, for every $r<t^*-\q x$ we find
\begin{eqnarray*}
\H^n(\S_f\cap\C_{x,r})&=&\int_{\q x-r}^{\q x+r}\H^{n-1}\Big(\{f<s\}\cap \D_{\p x,r}\Big)\,ds
\\
&\le& 2r\,\H^{n-1}\Big(\{f<t^*\}\cap \D_{\p x,r}\Big)=o(r^{n})\,.
\end{eqnarray*}
This proves \eqref{brigthon a}, and \eqref{brigthon b} follows similarly. As a consequence, $\pae \S_f\subset\G_f$, from which \eqref{b} follows, as well as that $\pae\S_f\cap(S_f\times\R)\subset \G_f\cap(S_f\times\R)$. This last inclusion implies \eqref{a}, as
\[
\Big(\G_f\cap(S_f\times\R)\Big)\setminus\Gv_f=\Big\{(z,f^\wedge(z)):z\in S_f\Big\}\cup\Big\{(z,f^\vee(z)):z\in S_f\Big\}\,,
\]
is $\H^{n-1}$-negligible (indeed, it projects twice over the countably $\H^{n-2}$-rectifiable set $S_f$).

\medskip

\noindent {\it Step three}\,: Let now $N_f$ be as in Lemma \ref{lemma fico}. We claim that, if $z\in S_f\setminus N_f$ and $f^\wedge(z)<t<f^\vee(z)$ (so that $z\in\pa^J\{f>t\}$ for every such $t$, with constant jump direction $\nu(z)\in S^{n-1}\cap\R^{n-1}$), then $(z,t)\in\pa^J\S_f$ with jump direction given by $(-\nu(z),0)$; in particular,
\begin{equation}\label{tizio}
  \Gv_f\cap\Big((S_f\setminus N_f)\times\R\Big)\subset\pa^J\S_f\,.
\end{equation}
Indeed, if $t_0,t_1\in I$ are such that $f^\wedge(z)<t_0<t<t_1<f^\vee(z)$, then for $r$ small enough,
\begin{eqnarray*}
  &&\H^n\Big(\Big(\S_f\Delta H_{(z,t),(-\nu(z),0)}^+\Big)\cap\C_{(z,t),r}\Big)
  \\
  &=&\int_{t-r}^{t+r}\H^{n-1}(\D_{z,r}\cap H_{z,-\nu(z)}^-\cap\{f<s\})\,+\H^{n-1}(\D_{z,r}\cap H_{z,-\nu(z)}^+\cap\{f\ge s\})\,ds
  \\
  &\le& 2r\,\H^{n-1}(\D_{z,r}\cap H_{z,\nu(z)}^+\cap\{f<t_1\})
  +2r\,\H^{n-1}(\D_{z,r}\cap H_{z,\nu(z)}^-\cap\{f\ge t_0\})=o(r^{n})\,,
\end{eqnarray*}
as $\{f<t_1\}_{z,r}\toloc H_{z,\nu(z)}^-$ and $\{f\ge t_0\}_{z,r}\toloc H_{z,\nu(z)}^+$. We conclude by Federer's theorem.

\medskip

\noindent {\it Step four}\,: By \eqref{b}, \eqref{brigthon a}, \eqref{brigthon b}, and by Federer's theorem we deduce \eqref{bonn1}. By \eqref{a}, \eqref{tizio}, and by Federer's theorem, we prove \eqref{bonn2}. Finally, a last application of Federer's theorem allows to deduce \eqref{Sf1} and \eqref{Sf0} from \eqref{bonn1}, \eqref{bonn2}, \eqref{brigthon a}, and \eqref{brigthon b}.
\end{proof}

Recall that, if $M\subset\R^n$ and $z\in\R^{n-1}$, then $M_z=\{t\in\R:(z,t)\in M\}$. As a corollary of Proposition \ref{proposition fine bv} we thus find the following statement.

\begin{corollary}\label{remark F1}
  If $f\in GBV_*(\R^{n-1})$ and $N_f$ is defined as in Lemma \ref{lemma fico}, then for every $z\in S_f\setminus N_f$ we have
  \begin{eqnarray}\label{inclusioni tizio}
    (\Gv_f)_z=(f^\wedge(z),f^\vee(z))&\subset&\Big(\pa^J\S_f\cap(S_f\times\R)\Big)_z
    \\\nonumber
    &\subset&\Big(\pae \S_f\cap(S_f\times\R)\Big)_z
    \subset[f^\wedge(z),f^\vee(z)]\,.
  \end{eqnarray}
  In particular, for every Borel set $A\subset S_f$ we have
  \[
  P_\g(\S_f;A\times\R)=\int_A\,\int_{f^\wedge(z)}^{f^\vee(z)}\,d\H^1_\g(t)\,d\H^{n-2}_\g(z)\,.
  \]
\end{corollary}

\begin{proof}
  The first inclusion in \eqref{inclusioni tizio} follows immediately from \eqref{tizio}, while the second inclusion is immediate from \eqref{inclusioni frontiere}. The third inclusion follows of course from $\pae \S_f\subset\G_f$. Finally, since $S_f$ is countably $\H^{n-2}$-rectifiable, \eqref{federer 3.2.23} implies $\H^{n-1}\llcorner (S_f\times\R)=(\H^{n-2}\llcorner S_f)\times\H^1$. Thus, is $A$ is a Borel set with $A\subset S_f$, then by \eqref{inclusioni tizio} we find
  \begin{eqnarray*}
    P_\g(\S_f;A\times\R)&=&\H^{n-1}_\g(\pae \S_f\cap(A\times\R))
    =\int_A\,\H^1_\g((\pae \S_f)_z)\,d\H^{n-2}_\g(z)
    \\
    &=&\int_A\,\int_{f^\wedge(z)}^{f^\vee(z)}\,d\H^1_\g(t)\,d\H^{n-2}_\g(z)\,,
  \end{eqnarray*}
  where the tensorization property of $e^{-|x|^2/2}$ was also taken into account.
\end{proof}

\subsection{Proof of Theorem \ref{thm gauss}: (ii) implies (i)}\label{section gauss ii implica i} In this section we present the proof of the implication $(ii)\Rightarrow(i)$ in Theorem \ref{thm gauss}. At the end of the proof we collect some examples and remarks that should justify the rather involved technical argument we adopt.

\begin{proof}[Proof of Theorem \ref{thm gauss}, (ii) implies (i)] {\it Overview}\,: We let $v:\R^{n-1}\to [0,1]$ be a Lebesgue measurable function such that $P_\g(F[v])<\infty$ (and, therefore, $\{F[v],g(F[v])\}\subset\M(v)$). If we define $f:\R^{n-1}\to\R\cup\{\pm\infty\}$ as $f(z)=\Psi(v(z))$, $z\in\R^{n-1}$, then
\[
F[v]=\S_f=\mbox{epigraph of $f$}\,.
\]
We shall set for brevity $F=F[v]$. Since $F$ has finite Gaussian perimeter, it turns out that $F$ is of locally finite perimeter, and thus, by Proposition \ref{prop f GBV}, that $f\in GBV_*(\R^{n-1})$. Up to redefine $v$ on a $\H^{n-1}$-negligible set, we can also assume that $v$ is Borel measurable. (As noticed in the introduction, Theorem \ref{thm gauss} is stable under modifications of $v$ over $\H^{n-1}$-negligible sets.) We now consider the Borel set
\[
G=\{z\in\R^{n-1}:0<v(z)<1\}=\{z\in\R^{n-1}:f(z)\in\R\}\,,
\]
and assume that
\begin{equation}
  \label{proof ii implica i not disconnect}
  \mbox{$\{v^\wedge=0\}\cup\{v^\vee=1\}$ does not essentially disconnect $G$}.
\end{equation}
We want to prove that, if $E$ is a $v$-distributed set such that
\begin{equation}
  \label{proof ii implica i perimetri uguali}
  P_\g(E)=P_\g(F)\,,
\end{equation}
then either $\H^n(E\Delta F)=0$ or $\H^n(E\Delta g(F))=0$, where $g$ denotes the reflection with respect to $\R^{n-1}$, $g(x)=(\p x,-\q x)$, $x\in\R^n$. To this end, let us set as usual $E_z=\{t\in\R:(z,t)\in E\}$ for $z\in\R^{n-1}$, and set
\begin{eqnarray*}
  G_+&=&\Big\{z\in G:\H^1\Big(E_z\Delta (f(z),\infty)\Big)=0\Big\}\,,
  \\
  G_-&=&\Big\{z\in G:\H^1\Big(E_z\Delta (-\infty,-f(z))\Big)=0\Big\}\,,
  \\
  G_1&=&\{v=1\}=\Big\{z\in \R^{n-1}:\H^1\Big(E_z\Delta \R\Big)=0\Big\}\,,
  \\
  G_0&=&\{v=0\}=\Big\{z\in \R^{n-1}:\H^1(E_z)=0\Big\}\,.
\end{eqnarray*}
By Theorem \ref{thm cfmp1} we find that
\begin{equation}
\label{proof ii implica i E}
      E=_{\H^n}\Big(F\cap\Big((G_+\cup G_1)\times\R\Big)\Big)\cup \Big(g(F)\cap(G_-\times\R)\Big)\,,
\end{equation}
as well as that $\{G_+,G_-,G_1,G_0\}$ is a partition of $\R^{n-1}$ modulo $\H^{n-1}$, and that $\{G_+,G_-\}$ is a partition of $G$ modulo $\H^{n-1}$, where this last condition means
\[
\H^{n-1}(G\Delta(G_+\cup G_-))=0\,,\qquad\H^{n-1}(G_+\cap G_-)=0\,.
\]
Clearly, $G=\{0<v<1\}$, $G_1=\{v=1\}$, and $G_0=\{v=0\}$ are Borel sets, as $v$ is a Borel function. Notice that also $G_+$ and $G_-$ are Lebesgue measurable sets. Indeed, if we define $\b:\R^{n-1}\to\R$ as
  \[
  \b(z)=\left\{\begin{array}
    {l l}
    \frac1{v(z)}\int_{E_{z}}t\,d\g_1(t)\,,&z\in\,\{0<v\le1\}\,,
    \\
    0\,,&z\in\{v=0\}\,,
  \end{array}\right .
  \]
  (so that  $\b(z)$ is the Gaussian barycenter of $E_{z}$), then, by Fubini's theorem, $\b$ is a Lebesgue measurable function. At the same time, a simple computation shows that
  \[
  \b(z)=\frac1{\sqrt{2\pi}}\,\Big(1_{G_+}(z)\,\frac{e^{-f(z)^2/2}}{v(z)}
  -1_{G_-}(z)\,\frac{e^{-f(z)^2/2}}{v(z)}\Big)\,,\qquad\forall z\in G\cup G_1\,,
  \]
  so that $G_+=\{\b>0\}$ and $G_-=\{\b<0\}$. Thus, both $G_+$ and $G_-$ are Lebesgue measurable sets. We now look back at \eqref{proof ii implica i E}, and notice that $\H^n(E\Delta F)\,\H^n(E\Delta g(F))=0$ if and only if $\H^{n-1}(G_+)\H^{n-1}(G_-)=0$. We thus argue by contradiction, and assume that rigidity fails because of $E$, which amounts in asking that
  \begin{equation}
    \label{proof ii implica i contradiction}
      \H^{n-1}(G_+)\,\H^{n-1}(G_-)>0\,.
  \end{equation}
  In other words, $\{G_+,G_-\}$ is a non-trivial Lebesgue measurable partition of $G$. Hence, thanks to \eqref{proof ii implica i not disconnect}, by Borel regularity of the Lebesgue measure, and since $\pae A=\pae B$ if $A,B\subset\R^{n-1}$ with $\H^{n-1}(A\Delta B)=0$, we find that
  \begin{equation}
    \label{h2}
      \H^{n-2}\Big(\Big(G^{(1)}\cap\pae G_+\cap\pae G_-\Big)\setminus\Big(\{v^\wedge=0\}\cup\{v^\vee=1\}\Big)\Big)>0\,.
  \end{equation}
  Comparing \eqref{proof ii implica i E} and \eqref{h2} we see that $E$ is obtained by reflecting $F$ across a region of {\it non-trivial $\H^{n-2}$ measure} where the sections of $F$ are {\it neither negligible nor equivalent to $\R$}: correspondingly, we expect Gaussian perimeter to be increased in this operation, that is, we expect \eqref{proof ii implica i E} and \eqref{h2} to imply $P_\g(E)>P_\g(F)$, thus contradicting \eqref{proof ii implica i perimetri uguali}. The main difficulty in proving that this actually happens relies on the fact that the set $G^{(1)}\cap\pae G_+\cap\pae G_-$ may not have a reasonable metric structure, that is, it may fail to be countably $\H^{n-2}$-rectifiable. (Example \ref{example G per finito} shows that $G$ may fail to be of locally finite perimeter. Example \ref{remark koch} shows that $G^{(1)}\cap\pae G_+\cap\pae G_-$ may fail to be countably $\H^{n-2}$-rectifiable even if $v\in \Lip(\R^{n-1};[0,1])$.) We shall avoid this difficulty by showing the existence of a countably $\H^{n-2}$-rectifiable set $\S$ such that
  \[
  \S\subset\Big(G^{(1)}\cap\pae G_+\cap\pae G_-\Big)\setminus\Big(\{v^\wedge=0\}\cup\{v^\vee=1\}\Big)\,,\qquad \H^{n-2}(\S)>0\,.
  \]
  We shall then deduce that, as simple drawings suggest, $P_\g(E;\S\times\R)>P_\g(F;\S\times\R)$. Finally, by taking into account that $P_\g(E;A\times\R)\ge P_\g(F;A\times\R)$ for every Borel set $A\subset\R^{n-1}$, we shall find $P_\g(E)>P_\g(F)$. We divide this argument in nine steps.

\medskip

\noindent {\it Step one}\,: We use the information that $E$ is a set of locally finite perimeter to deduce that for every $k\in\N$ the function $u_k:\R^{n-1}\to\R$ defined as
$$
u_k=(k-|f|)\,1_{\{|f|<k\}}\,\Big(1_{G_+}-1_{G_-}\Big)\in BV_{loc}(\R^{n-1})\,.
$$
Indeed, if we take into account \eqref{proof ii implica i E} and repeat the argument in the proof of Proposition \ref{prop f GBV} with $E$ in place of $F=\S_f$, then we find
\begin{eqnarray}\nonumber
      P(E;K\times I)&\ge&
  \int_{G_+}\nabla'\vphi(z)\,dz\int_{f(z)}^\infty\,\psi'(t)dt+\int_{G_-}\nabla'\vphi(z)\,dz\int^{-f(z)}_{-\infty}\,\psi'(t)dt
  \\    \label{test psi}
  &&+\int_{G_1}\nabla'\vphi(z)\,dz\int_{-\infty}^\infty\,\psi'(t)dt\,,
\end{eqnarray}
  whenever $\vphi\in C^1_c(\R^{n-1})$ with $\spt\vphi\subset K\cc\R^{n-1}$ and $|\vphi|\le 1$, $\psi:\R\to\R$ is a Lipschitz function with $\spt\,\psi'\subset I\cc\R$ and $\Lip(\psi)\le 1$. If we apply \eqref{test psi} with $\psi$ defined by $\psi(t)=k$ for $|t|>k$ and $\psi(t)=|t|$ for $|t|\le k$,
  then we deduce our assertion by exploiting the relations (valid for every $a\in\R$)
  \[
  \int_a^\infty\psi'=(k-|a|)\,1_{(-k,k)}(a)\,,\qquad\int^{-a}_{-\infty}\psi'=-(k-|a|)\,1_{(-k,k)}(a)\,,\qquad
  \int_{-\infty}^{\infty}\psi'=0\,.
  \]

\noindent {\it Step two}\,: We show that, for every $k \in \mathbb{N}$,
$$
\Big\{ |f|^{\vee} < \frac{k}2\Big\} \cap G_+^{(1)} \subset \Big\{u_k^\wedge > \frac{k}2\Big\} \cap G_+^{(1)}\,.
$$
It suffices to prove that, if $z \in \{ |f|^{\vee} < k/2\} \cap G_+^{(1)}$ and $\e<(k/2)-|f|^\vee(z)$, then
\[
\theta( \{u_k < s \} , z ) = 0\,,\qquad\forall s < \frac{k}2+\e\,.
\]
Indeed, thanks to \eqref{dens2}, we have $\{ |f|^{\vee} < k/2 \} \subset \{ |f| < k/2 \}^{(1)}$. Thus, for every such $s$,
\begin{align*}
&\H^{n-1}\Big(\D_{z,r} \cap \{u_k < s \}\Big)
\\
\mbox{{\small ($z\in G_+^{(1)}$)}}\qquad& = \H^{n-1}\Big( \D_{z,r} \cap \{u_k < s \} \cap G_+ \Big) + o (r^{n-1}) \\
\mbox{{\small ($z\in \{ |f| < k/2 \}^{(1)}$)}}\qquad&= \H^{n-1}\Big(\D_{z,r} \cap \{u_k < s \} \cap \{ |f| < k/2 \} \cap G_+ \Big) + o (r^{n-1})
\\
&= \H^{n-1}\Big(\D_{z,r} \cap \{ k-|f| < s \} \cap \{ |f| < k/2 \} \cap G_+ \Big) + o (r^{n-1}) \\
&\le \H^{n-1}\Big(\D_{z,r} \cap \{ k - s < |f|\}\Big) + o (r^{n-1}) = o (r^{n-1})\,,
\end{align*}
where the last identity follows by definition of $|f|^\vee$ since $k-s> k/2-\e>|f|^\vee(z)$.

\medskip

\noindent {\it Step three}\,: We set
 \[
 \Sigma_k = \pae G_+\cap\pae G_-\cap\Big\{ - \frac{k}2 < f^\wedge  \leq f^\vee  < \frac{k}2 \Big\}^{(1)}\,,\qquad k\in\N\,,
 \]
 and prove that
\begin{equation}
   \label{uksalta}
 \Sigma_k \subset \left\{ u_k^\vee \ge\frac{k}2 \right\} \cap \left\{ u_k^\wedge \le-\frac{k}2 \right\}\,,\qquad\forall k\in\N\,.
 \end{equation}
To show this, we start noticing that for every $z\in\S_k$ we have
 \begin{align}\nonumber
\H^{n-1}\Big(\D_{z,r}\cap \Big\{u_k>\frac{k}2\Big\}\Big) &= \H^{n-1}\Big(\D_{z,r}\cap\Big\{u_k>\frac{k}2\Big\}^{(1)}\Big)
\\\nonumber&\ge \H^{n-1}\Big(\D_{z,r}\cap \Big\{u_k>\frac{k}2\Big\}^{(1)} \cap G_+^{(1)}\Big) \\
& \ge \H^{n-1}\Big(\D_{z,r}\cap \Big\{u_k^\wedge >\frac{k}2\Big\} \cap G_+^{(1)}\Big)\,,\label{bus}
 \end{align}
where the last inequality follows from \eqref{dens3}. Now, by step two and by \eqref{dens1},
$$
\Big\{u_k^\wedge >\frac{k}2\Big\} \cap G_+^{(1)}
\supset \Big\{ |f|^{\vee} < \frac{k}2\Big\} \cap G_+^{(1)}
=\Big\{ - \frac{k}2 < f^{\wedge} \leq f^{\vee} <\frac{k}2\Big\} \cap G_+^{(1)}\,,
$$
so that, by \eqref{bus},
 \begin{align*}
\H^{n-1}\Big(\D_{z,r}\cap \Big\{u_k>\frac{k}2\Big\}\Big)
&\ge \H^{n-1}\Big(\D_{z,r}\cap \Big\{ - \frac{k}2 < f^{\wedge} \leq f^{\vee} <\frac{k}2\Big\} \cap G_+^{(1)}\Big) \\
& = \H^{n-1}\Big(\D_{z,r}\cap G_+\Big) + o (r^{n-1}),
 \end{align*}
where in the last identity we have used the fact that $z \in \{k/2> f^\vee\ge f^\wedge>-k/2\}^{(1)}$. Since, by assumption, $z\in\pae  G_+$, we conclude that
$$
0<\theta^*( G_+ , z )\le \theta^*\Big( \Big\{u_k>\frac{k}2\Big\} , z \Big) \,,
$$
which in turn gives $u_k^\vee (z) \ge k/2$. One can prove analogously that $u_k^\wedge (z) \le - k/2$.

\medskip

\noindent {\it Step four}\,: We show that, for every $k \in \mathbb{N}$,
 $$
 \Sigma_k \text{ is locally } \mathcal{H}^{n-2}\text{-rectifiable}.
 $$
 From step three we have that $\Sigma_k \subset S_{u_k}$.
 Being $u_k \in BV_{loc} (\mathbb{R}^{n-1})$, this imples that $\Sigma_k$
 is countably $\H^{n-2}$-rectifiable, and we are only left to show that $\S_k$ is locally $\H^{n-2}$-finite.
 To this end, let $K \subset \mathbb{R}^{n-1}$ be a compact set; since
 $$
 \Sigma_k =   \left[ \Sigma_k \cap (S_{u_k} \setminus J_{u_k}) \right] \cup \left( \Sigma_k \cap J_{u_k} \right)
 $$
 and $\mathcal{H}^{n-2} (S_{u_k} \setminus J_{u_k}) = 0$, we have
 $$
 \mathcal{H}^{n-2} (\Sigma_k \cap K) =  \mathcal{H}^{n-2} (\Sigma_k \cap J_{u_k} \cap K)\in[0,\infty]\,.
 $$
 By step three and since $u_k \in BV_{loc} (\mathbb{R}^{n-1})$,
 \begin{align*}
k\, \mathcal{H}^{n-2} (\Sigma_k \cap J_{u_k} \cap K)
\leq \int_{\Sigma_k \cap J_{u_k} \cap K} ( u_k^\vee - u_k^\wedge ) \, d \mathcal{H}^{n-2}
\le |D^j u_k| (K)\,.
 \end{align*}
 Thus, if $K\subset\R^{n-1}$ is compact and $k\in\N$, then $\H^{n-2}(K\cap\Sigma_k)\le k^{-1}\, |D^j u_k| (K)<\infty$. This proves $\Sigma_k$ is locally $\H^{n-2}$-finite.

  \medskip

  \noindent {\it Step five}\,: We are now going to deduce from \eqref{h2} that, for $k$ sufficiently large, we have
  \begin{equation}
    \label{step six}
    \H^{n-2}(\S_k)>0\,.
  \end{equation}
  We start proving the following identity,
  \begin{eqnarray}\label{austin2013 0}
  \bigcup_{k \in \mathbb{N}} \Sigma_k =\Big(\pae  G_+ \cap \pae  G_-\Big) \setminus \Big(\{f^\vee=\infty\}\cup \{f^\wedge=-\infty\}\Big)\,.
  \end{eqnarray}
  Indeed, by definition of $\S_k$, and by repeatedly applying \eqref{dens2} and \eqref{dens3},
  \begin{eqnarray}\nonumber
     \Sigma_k &=& \pae G_+\cap\pae G_-\cap\Big\{ - \frac{k}2 < f^\wedge  \leq f^\vee  < \frac{k}2 \Big\}^{(1)}
     \\\nonumber
     &\subset&\pae G_+\cap\pae G_-\cap\Big\{ - \frac{k}2 \le f^\wedge  \leq f^\vee  \le \frac{k}2 \Big\}
     \\\nonumber
     &\subset&\pae G_+\cap\pae G_-\cap\Big\{ - \frac{k+1}2 <f^\wedge  \leq f^\vee  <\frac{k+1}2 \Big\}
     \\\label{catenaSk}
     &\subset&\pae G_+\cap\pae G_-\cap\Big\{ - \frac{k+1}2 <f^\wedge  \leq f^\vee  <\frac{k+1}2 \Big\}^{(1)}=\S_{k+1}\,,
  \end{eqnarray}
  from which \eqref{austin2013 0} immediately follows. Since $f=\Psi(v)$ with $\Psi$ continuous and decreasing, and thanks to \eqref{continuous and decreasing}, we have $\{f^\vee=\infty\}=\{v^\wedge=0\}$ and $\{f^\wedge=-\infty\}=\{v^\vee=1\}$, so that \eqref{austin2013 0} is equivalent to
  \begin{equation}
    \label{inclusion filippo}
      \bigcup_{k \in \mathbb{N}} \Sigma_k =\Big( \pae  G_+ \cap \pae  G_-\Big) \setminus \left( \{ v^{\wedge} = 0 \} \cup \{  v^{\vee} = 1 \} \right)\,.
  \end{equation}
  Finally, by \eqref{inclusion filippo}, \eqref{catenaSk},  and \eqref{h2}, we find
  $$
  \lim_{k \to \infty} \mathcal{H}^{n-2} ( \Sigma_k ) =  \mathcal{H}^{n-2} \Big(\Big(\pae  G_+
    \cap \pae  G_-\Big) \setminus \Big( \{ v^{\wedge} = 0 \} \cup \{  v^{\vee} = 1 \} \Big) \Big) > 0\,.
  $$

  \medskip

\noindent {\it Step six}\,: We show here that, if $W\subset\S_k$ is a Borel set, then
\begin{equation}
  \label{brisket}
P_\g(F;W\times\R)=\int_{W}\,d\H^{n-2}_\g(z)\,\int_{f^\wedge(z)}^{f^\vee(z)}\,d\H^1_\g\,.
\end{equation}
Indeed, \eqref{brisket} follows immediately by Corollary \ref{remark F1} provided $W\subset S_f$. Since the right-hand side of \eqref{brisket} is trivially equal to zero if $W\subset S_f^c$, we are left to prove that
\[
P_\g(F;(\S_k\cap S_f^c)\times\R)=0\,.
\]
To this end, we notice that, by Proposition \ref{proposition fine bv},
\[
\pae F\cap(S_f^c\times\R)\subset_{\H^{n-1}}\Big\{x\in\R^n:\p x\in S_f^c\,,\q x=f^\wedge(\p x)=f^\vee(\p x)\Big\}\,.
\]
If $L$ denotes the set on the right-hand side of this last inclusion, then $\H^0(L_z)=1$ for every $z\in S_f^c$. As $\S_k$ is countably $\H^{n-2}$-rectifiable, by \eqref{federer 3.2.23} we find that
\begin{eqnarray*}
P_\g(F;(S_f^c\cap\S_k)\times\R)&=&\H^{n-1}_\g\Big(\pae F\cap\Big((S_f^c\cap\S_k)\times\R\Big)\Big)
\\
&\le&\H^{n-1}_\g\Big(L\cap\Big((S_f^c\cap\S_k)\times\R\Big)\Big)
=\int_{S_f^c\cap\S_k}\,\H^1_\g(L_z)\,d\H^{n-2}_\g(z)=0\,.
\end{eqnarray*}
We have thus completed the proof of \eqref{brisket}.

\medskip

\noindent {\it Step seven}\,: We show that, if $z\in \S_k\setminus N_{u_k}$ (with $N_{u_k}$ defined as in Lemma \ref{lemma fico}), then there exists $\nu\in S^{n-1}\cap\R^{n-1}$ such that
\begin{eqnarray}
  \label{jeez1}
  &&(G_+)_{z,r}\toloc H_{0,\nu}^+\,,\qquad\hspace{0.5cm} (G_-)_{z,r}\toloc H_{0,\nu}^-\,,
  \\\label{jeez2}
  &&\{u_k>t\}_{z,r}\toloc H_{0,\nu}^+\,,\qquad\forall t\in(u_k^\wedge(z),u^\vee_k(z))\,.
\end{eqnarray}
By \eqref{austin2013 0} and since $\H^{n-2}(N_{u_k})=0$,  this will imply in particular that
\begin{equation}
  \label{austin2013}
  \S_k\subset_{\H^{n-2}}\pa^JG_+\cap\pa^JG_-\cap\{|f|^\vee<\infty\}\,.
\end{equation}
We first recall that, by Lemma \ref{lemma fico}, if $z\in S_{u_k}\setminus N_{u_k}$, then there exists $\nu=\nu(z)\in S^{n-2}$ such that \eqref{jeez2} holds true.
Now, we easily find that
\begin{eqnarray*}
\{u_k>t\}=G_+\cap\{|f|<k-t\}\,,\qquad\forall t>0\,,
\end{eqnarray*}
which in particular gives,
\[
z\in \bigcap_{0<t<u^\vee_k(z)}\,\pa^J\Big(G_+\cap\Big\{|f|<k-t\Big\}\Big)\,.
\]
Since, by \eqref{uksalta}, $u_k^\vee(z)\ge k/2$ for every $z\in\S_k$, for $\e$ small enough we find that
\begin{eqnarray*}
  \S_{k}\setminus N_{u_k}\subset\,\pa^J\Big(G_+\cap\Big\{|f|<k-\Big(\frac{k}2-\e\Big)\Big\}\Big)
  =\pa^J\Big(G_+\cap\Big\{|f|<\frac{k}2+\e\Big\}\Big)\,.
\end{eqnarray*}
Taking now into account that $\pa^J(A\cap B)\cap B^{(1)}\subset (\pa^JA)\cap B^{(1)}$, we thus find
\begin{eqnarray*}
  \Big(\S_{k}\setminus N_{u_k}\Big)\cap\Big\{|f|<\frac{k}2+\e\Big\}^{(1)}\subset\pa^JG_+\,.
\end{eqnarray*}
Finally, since $\S_k\subset\{|f|<(k/2)+\e\}^{(1)}$, we conclude that $\S_k\setminus N_{u_k}\subset \pa^JG_+$. One proves analogously the inclusion in $\pa^JG^-$.

\medskip

\noindent {\it Step eight}\,: We have so far proved that, if $k$ is large enough, then $\S_k$ is a locally $\H^{n-2}$-rectifiable set in $\R^{n-1}$, with $\H^{n-2}(\S_k)>0$, and $\S_k\subset\pa^JG_+\cap\pa^JG-\cap\{|f|^\vee<\infty\}$ (modulo $\H^{n-2}$). Moreover, we have computed the Gaussian perimeter of $F$ above $\S_k$. We now want to compute $P_\g(E;\S_k\times\R)$, in order to show that this last quantity is strictly larger than $P_\g(F;\S_k\times\R)$. To this end, it is convenient to divide $\S_k$ into two parts, defined by the sets $\Pi_+$ an $\Pi_-$ introduced in this and in the following step. Precisely, we start this conclusive part of our argument by considering the set $\Pi_+$ of those
\[
z\in\pa^JG_+\cap\pa^JG_-\cap\{|f|^\vee<\infty\}\cap(S_f^c\cup J_f)\,,
\]
such that, for some $\nu\in S^{n-1}\cap \R^{n-1}$,
\begin{eqnarray}
  \label{blll}
  &&(G_+)_{z,r}\toloc H_{0,\nu}^+\,,\qquad (G_-)_{z,r}\toloc H_{0,\nu}^-\,,
  \\\label{blll2}
  &&\{f>s\}\toloc H_{0,\nu}^+\,,\qquad\mbox{if $z\in J_f$ and $s\in(f^\wedge(z),f^\vee(z))$}\,.
\end{eqnarray}
We want to characterize $(\pa^JE)_z$ for $z\in\Pi_+$, by showing that
\begin{eqnarray}\label{gradi1}
(\pa^JE)_z=_{\H^1}(-\infty,-f^\wedge(z))\cup(f^\vee(z),\infty)\,,\qquad \forall z\in\Pi_+\cap\{f^\vee\ge-f^\wedge\}\,,
\\\label{gradi2}
(\pa^JE)_z=_{\H^1}(-\infty,f^\vee(z))\cup(-f^\wedge(z),\infty)\,,\qquad \forall z\in\Pi_+\cap\{f^\vee\le -f^\wedge\}\,.
\end{eqnarray}
In particular, we shall prove that, if $z\in\Pi_+$ and $f^\vee(z)\ge -f^\wedge(z)$, then
\begin{eqnarray}\label{gigi1}
(z,t)\in\pa^JE\,,&&\qquad \forall t\in\Big(-\infty,-f^\wedge(z)\Big)\cup\Big(f^\vee(z),\infty\Big)\,,
\\\label{gigi2}
(z,t)\in E^{(0)}\subset\R^n\setminus\pae E\,,&&\qquad\forall t\in(-f^\wedge(z),f^\vee(z))\,,
\end{eqnarray}
(so that \eqref{gradi1} holds true, see Figure \ref{fig pipiu}),
\begin{figure}
  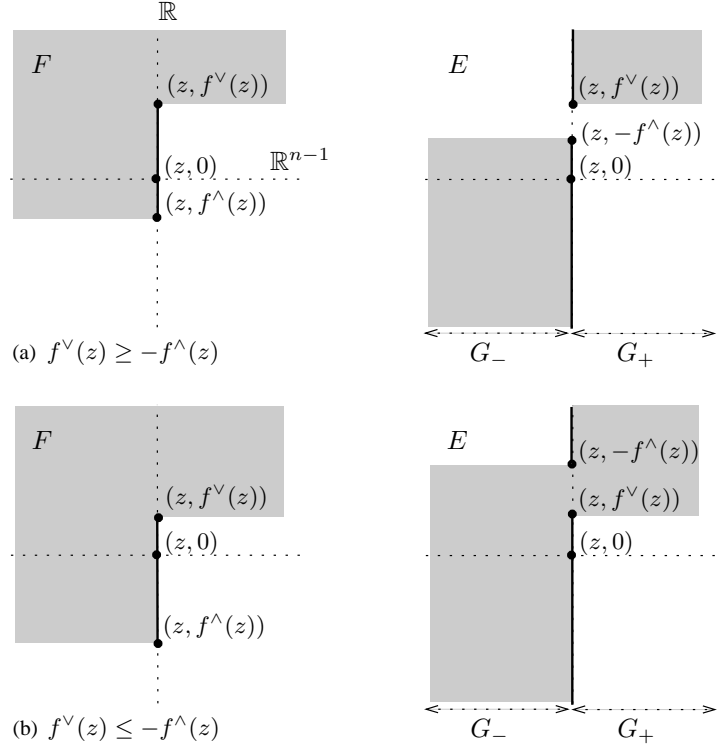\caption{{\small In panel (a) we consider the case when $z\in\Pi_+$ and $f^\vee(z)\ge-f^\wedge(z)$. In this case we must have $f^\vee(z)\ge 0$, while, of course, $f^\wedge(z)$ has arbitrary sign. Moreover, $(\pae E)_z$ is $\H^1$-equivalent to $(-\infty,-f^\wedge(z))\cup(f^\vee(z),\infty)$, see \eqref{gigi1}, and $(-f^\wedge(z),f^\vee(z))$ is $\H^1$-equivalent to $(E^{(0)})_z$, see \eqref{gigi2}. In panel (b) we consider the complementary case when $z\in\Pi_+$ and $f^\vee(z)\le-f^\wedge(z)$. In this case $(\pae E)_z$ is $\H^1$-equivalent to $(-\infty,f^\vee(z))\cup(-f^\wedge(z),\infty)$, see \eqref{gigi1xx}, while $(f^\vee(z),-f^\wedge(z))$ is $\H^1$-equivalent to $(E^{(1)})_z$, see \eqref{gigi2xx}. In both cases, of course, $(\pae F)_z$ is $\H^1$-equivalent to $(f^\wedge(z),f^\vee(z))$.}}\label{fig pipiu}
\end{figure}
while, if $z\in\Pi_+$ and $f^\vee(z)\le -f^\wedge(z)$, then
\begin{eqnarray}\label{gigi1xx}
(z,t)\in\pa^JE\,,&&\qquad \forall t\in\Big(-\infty,f^\vee(z)\Big)\cup\Big(-f^\wedge(z),\infty\Big)\,,
\\\label{gigi2xx}
(z,t)\in E^{(1)}\subset\R^n\setminus\pae E\,,&&\qquad\forall t\in(f^\vee(z),-f^\wedge(z))\,.
\end{eqnarray}
(thus proving \eqref{gradi2}, see, once again, Figure \ref{fig pipiu}). Before entering into the proof of \eqref{gigi1}, \eqref{gigi2}, \eqref{gigi1xx}, and \eqref{gigi2xx}, let us notice that \eqref{gradi1} and \eqref{gradi2} imply that
\begin{equation}
  \label{aleee}
  (\pa^JE)_z=_{\H^1} (-\infty,a(z))\cup(b(z),\infty)\,,\qquad\forall z\in\Pi_+\,,
\end{equation}
where we have set
\begin{eqnarray}\label{abz}
a(z)=\min\Big\{-f^\wedge(z),f^\vee(z)\Big\}\,,
\qquad
b(z)=\max\Big\{-f^\wedge(z),f^\vee(z)\Big\}\,.
\end{eqnarray}
We shall now provide the details of the proof of \eqref{gigi1}, noticing that \eqref{gigi2}, \eqref{gigi1xx}, and \eqref{gigi2xx}, can be proved by entirely analogous arguments. Let us thus consider $z\in\Pi_+$ with $f^\vee(z)\ge -f^\wedge(z)$, and notice that, necessarily, $f^\vee(z)\ge(f^\vee(z)+f^\wedge(z))/2\ge 0$. We now consider two separate cases.

\smallskip

\noindent {\it Proof of \eqref{gigi1} when $t>f^\vee(z)$}\,: Let $r_*>0$ be such that $t-r_*>f^\vee(z)$, so that
\begin{eqnarray}\label{rstar}
  \{f<s\}_{z,r}\toloc\R^{n-1}\,,\qquad \{f<-s\}_{z,r}\toloc\emptyset\,,\qquad\forall s\in[t-r_*,t+r_*]\,,
\end{eqnarray}
thanks to the fact that $f^\vee(z)\ge 0$. Since $z\in G_+^{(1/2)}\cap G_-^{(1/2)}\subset G_1^{(0)}\cap G_0^{(0)}$, we have that
\begin{equation}
  \label{houston1}
  \H^n\Big(\C_{(z,t),r}\cap\Big((G_1\cup G_0)\times\R\Big)\Big)=o(r^n)\,;
\end{equation}
moreover, if $r<r_*$, then by \eqref{rstar} and by \eqref{blll} we have
\begin{eqnarray}  \label{houston2}
  \H^n\Big(E\cap\C_{(z,t),r}\cap(G_-\times\R)\Big)&=&\int_{t-r}^{t+r}\H^{n-1}\Big(G_-\cap \{f<-s\}\cap\D_{z,r}\Big)\,ds\hspace{1cm}
  \\\nonumber
  &\le& 2r\,\H^{n-1}\Big(\{f<-(t-r_*)\}\cap\D_{z,r}\Big)=o(r^n)\,,
\end{eqnarray}
as well as,
\begin{eqnarray}\label{houston3}
  \H^n\Big(H_{(z,t),(\nu,0)}^+\cap\C_{(z,t),r}\cap(G_-\times\R)\Big)&=&2r\,\H^{n-1}\Big(H_{z,\nu}^+\cap G_-\cap\D_{z,r}\Big)=o(r^n)\,.
\end{eqnarray}
Hence, by \eqref{houston1}, \eqref{houston2}, and \eqref{houston3}, and by taking again into account \eqref{rstar} and \eqref{blll},
\begin{eqnarray*}
  &&\H^n\Big(\Big(E\Delta H_{(z,t),(\nu,0)}^+\Big)\cap\C_{(z,t),r}\Big)
  \\
  &=&o(r^n)
  +\int_{t-r}^{t+r}\H^{n-1}\Big(\Big(H_{z,\nu}^+\Delta\Big(G_+\cap\{f<s\}\Big)\Big)\cap\D_{z,r}\Big)\,ds
  \\
  &=&o(r^n)
  +\int_{t-r}^{t+r}\H^{n-1}\Big(\Big(G_+\Delta H_{z,\nu}^+\Big)\cap\D_{z,r}\Big)\,ds=o(r^n)\,.
\end{eqnarray*}
This proves that if $t>f^\vee(z)$, then $(z,t)\in\pa^JE$ with $E_{(z,t),r}\toloc H_{(0,0),(\nu,0)}^+$.

\smallskip

\noindent {\it Proof of \eqref{gigi1} when $t<-f^\wedge(z)$}\,: In the subcase that $t<-f^\vee(z)$, we immediately see (by symmetry) that $(z,t)\in\pa^JE$ with
\begin{equation}
  \label{still find}
  E_{(z,t),r}\toloc H_{(0,0),(\nu,0)}^-\,.
\end{equation}
In particular, if $z\in S_f^c$, this concludes the proof of \eqref{gigi1}. We are thus left to consider the case that $z\in J_f$ and $-f^\vee(z)< t<-f^\wedge(z)$. In this case, we still record the validity of \eqref{still find}, but this time, in the proof, we also have to take \eqref{blll2} into account: indeed, by \eqref{blll2} we have that
\[
  \{f<s\}_{z,r}\toloc H_{0,\nu}^-\,,\qquad\forall s\in(f^\wedge(z),f^\vee(z))\,,
\]
therefore, if $-f^\vee(z)< t<-f^\wedge(z)$ then there exists $r_*>0$ such that
\begin{equation}
  \label{fernando}
  \{f<-s\}_{z,r}\toloc H_{0,\nu}^-\,,\qquad\forall s\in[t-r_*,t+r_*]\,.
\end{equation}
We now notice that, since $t+r^*<-f^\wedge(z)\le f^\wedge(z)$, then $\{f<t+r_*\}_{z,r}\toloc \emptyset$, and therefore
\begin{eqnarray}  \label{houston2-1}
  \H^n\Big(E\cap\C_{(z,t),r}\cap(G_+\times\R)\Big)&\le&\int_{t-r}^{t+r}\H^{n-1}\Big(G_+\cap \{f<s\}\cap\D_{z,r}\Big)\,ds
  \\\nonumber
  &\le& 2r\,\H^{n-1}\Big(G_+\cap \{f<t+r_*\}\cap\D_{z,r}\Big)=o(r^n)\,.
\end{eqnarray}
By \eqref{blll}, we similarly have
\begin{eqnarray}
\label{houston3-1}
  \H^n\Big(H_{(z,t),(\nu,0)}^-\cap\C_{(z,t),r}\cap(G_+\times\R)\Big)=2r\,\H^{n-1}\Big(H_{z,\nu}^-\cap G_+\cap\D_{z,r}\Big)=o(r^n)\,.
\end{eqnarray}
By combining \eqref{houston2-1} and \eqref{houston3-1} with \eqref{houston1} (which holds true simply by $z\in G_+^{(1/2)}\cap G_-^{(1/2)}$), we thus find
\begin{eqnarray*}
  &&\H^n\Big(\Big(E\Delta H_{(z,t),(\nu,0)}^-\Big)\cap\C_{(z,t),r}\Big)
  \\&=&o(r^n)
  +\int_{t-r}^{t+r}\H^{n-1}\Big(\Big(H_{z,\nu}^-\Delta\Big(G_-\cap\{f<-s\} \Big)\Big)\cap\D_{z,r}\Big)\,ds
  \\&=&o(r^n)
  +\int_{t-r}^{t+r}\H^{n-1}\Big(G_-\cap\{f<-s\}\cap H_{z,\nu}^+\cap\D_{z,r}\Big)\,ds
  \\
  &&+\int_{t-r}^{t+r}\H^{n-1}\Big(\Big(H_{z,\nu}^-\setminus\Big(G_-\cap\{f<-s\}\Big)\Big)\cap\D_{z,r}\Big)\,ds
    \\&\le&o(r^n)
  +2r\,\H^{n-1}\Big(G_-\cap\{f<-(t-r_*)\}\cap H_{z,\nu}^+\cap\D_{z,r}\Big)
  \\
  &&+2r\,\H^{n-1}\Big(\Big(H_{z,\nu}^-\setminus\Big(G_-\cap\{f<-(t+r_*)\}\Big)\Big)\cap\D_{z,r}\Big)=o(r^n)\,,
\end{eqnarray*}
where in the last step we have also  \eqref{blll} and \eqref{fernando}. This concludes the proof of \eqref{gigi1}.

\medskip

\noindent {\it Step nine}\,: We finally find a contradiction. To this end, let us define $\Pi_-$ as the set of those
\[
z\in\pa^JG_+\cap\pa^JG_-\cap\{|f|^\vee<\infty\}\cap(S_f^c\cup J_f)\,,
\]
such that, for some $\nu\in S^{n-1}$,
\begin{eqnarray}
  \label{blll star}
  &&(G_+)_{z,r}\toloc H_{0,\nu}^+\,,\qquad (G_-)_{z,r}\toloc H_{0,\nu}^-\,,
  \\\label{blll2 star}
  &&\{f>s\}\toloc H_{0,\nu}^-\,,\qquad\mbox{if $z\in J_f$ and $s\in(f^\wedge(z),f^\vee(z))$}\,.
\end{eqnarray}
Let us now notice the following two facts. First, trivially,
\begin{eqnarray}\label{Sfc1}
(\Pi_+\cup\Pi_-)\cap S_f^c=\pa^JG_+\cap\pa^JG_-\cap\{|f|^\vee<\infty\}\cap S_f^c\,.
\end{eqnarray}
Second, since $J_f$ and $S_{u_k}$ are both countably $\H^{n-2}$-rectifiable sets, we have that $\nu_f=\pm\,\nu_{u_k}$ $\H^{n-2}$-a.e. on $J_f\cap S_{u_k}$, and thus by \eqref{jeez1} and \eqref{jeez2}, we find that
\begin{eqnarray}\label{Sfc2}
(\Pi_+\cup\Pi_-)\cap J_f\cap S_{u_k}=_{\H^{n-2}}\pa^JG_+\cap\pa^JG_-\cap\{|f|^\vee<\infty\}\cap J_f\cap S_{u_k}\,.
\end{eqnarray}
Since $\H^{n-2}(S_f\setminus J_f)=0$, we finally conclude that
\begin{eqnarray}\nonumber
  (\Pi_+\cup\Pi_-)\cap S_{u_k}=_{\H^{n-2}}\pa^JG_+\cap\pa^JG_-\cap\{|f|^\vee<\infty\}\cap S_{u_k}\,.
\end{eqnarray}
In particular, by  \eqref{step six} and \eqref{austin2013}, we may assume (up to replacing $E$ with $g(E)$) that
\[
\H^{n-2}(\S_k\cap\Pi_+)>0\,,
\]
for sufficiently large values of $k$. Since $\S_k$ is countably $\H^{n-2}$-rectifiable, by \eqref{federer 3.2.23} and by \eqref{aleee} we find
\begin{eqnarray*}
P_\g(E;(\S_k\cap\Pi_+)\times\R)&=&\int_{\S_k\cap\Pi_+}d\H^{n-2}_\g(z)\,\int_{(\pa^JE)_z}d\H^1_\g
\\
&=&
\int_{\S_k\cap\Pi_+}\,d\H^{n-2}_\g(z)\bigg(\int_{-\infty}^{a(z)}d\H^1_\g+\int_{b(z)}^\infty\,d\H^1_\g\bigg)\,,
\end{eqnarray*}
where $a$ and $b$ have been defined as in \eqref{abz}. Since $\H^1_\g(\R)=1$, we thus have
\begin{eqnarray*}
P_\g(E;(\S_k\cap\Pi_+)\times\R)=\int_{\S_k\cap\Pi_+}\Big(1-\g_1\Big(a(z),b(z)\Big)\Big)\,d\H^{n-2}_\g(z)\,\,,
\end{eqnarray*}
while, by \eqref{brisket},
\begin{eqnarray*}
P_\g(F;(\S_k\cap\Pi_+)\times\R)=\int_{\S_k\cap\Pi_+}\,\g_1\Big(f^\wedge(z),f^\vee(z)\Big)\,d\H^{n-2}_\g(z)\,.
\end{eqnarray*}
Since $P_\g(E;W\times\R)\ge P_\g(F;W\times\R)$ for every Borel set $W\subset\R^{n-1}$, by $P_\g(E)=P_\g(F)$ we find that
\[
P_\g(E;(\S_k\cap\Pi_+)\times\R)=P_\g(F;(\S_k\cap\Pi_+)\times\R)\,.
\]
This leads to a contradiction with the fact that $\H^{n-2}(\S_k\cap\Pi_+)>0$ and with the fact that the function
\[
\delta(\a,\b)=1-\g_1\Big(\min\{-\a,\b\},\max\{-\a,\b\}\Big)-\g_1(\a,\b)\,,\qquad \forall \beta\ge \a\,,
\]
is strictly positive on $\{(\a,\b)\in\R^2:\b\ge\a\}$. Indeed, if $-\a\le\b$, then we have
\begin{eqnarray*}
  \delta(\a,\b)&=&1-\g_1(-\a,\b)-\g_1(\a,\b)=1-\g_1(-\a,\b)-\g_1(-\b,-\a)
  \\
  &=&1-\g_1(-\b,\b)>0\,;
\end{eqnarray*}
if, instead, $-\a>\b$, then we have
\begin{eqnarray*}
  \delta(\a,\b)=1-\g_1(\b,-\a)-\g_1(\a,\b)=1-\g_1(\a,-\a)>0\,.
\end{eqnarray*}
This completes the proof of the implication $(ii)\Rightarrow(i)$.
\end{proof}

\begin{example}\label{example G per finito}
  {\rm It may happen that $v\in BV(\R^{n-1};[0,1])$ but $G=\{0<v<1\}$ is not of locally finite perimeter in $\R^{n-1}$. For example, if $n\ge 3$, take
  \[
  v(z)=\frac{|z|^2}{2}\sum_{h=1}^\infty\,1_{[1/(2h+1)^{1/(n-2)},1/(2h)^{1/(n-2)}]}(|z|)\,,\qquad z\in\R^{n-1}\,.
  \]
  In this case $G=\{0<v<1\}$ is not of locally finite perimeter, as
  \[
  \H^{n-2}(\D_r\cap\pae G)=\H^{n-2}(\D_r\cap\pa G)=(n-1)\om_{n-1}\sum_{h=h(r)}^\infty\frac1{2h}+\frac1{2h+1}=\infty\,,\qquad\forall r>0\,.
  \]
  At the same time $v\in BV(\R^{n-1};[0,1])$, as
  \[
  |Dv|(\R^{n-1})\le \sqrt{2}\H^{n-1}(G)+2(n-1)\om_{n-1}\sum_{h=1}^\infty\frac1{(2h)^{2/(n-2)}}\,\frac1{2h}<\infty\,.
  \]
  }
\end{example}


\begin{example}\label{remark koch}
  {\rm Consider an open equilateral triangle $T$ in $\R^2$, and define an increasing sequence of open sets $\{T_h\}_{h=0}^\infty$ by setting $T_0=T$; $T_1$ is obtained from $T_0$ by adding a copy of $T$ rescaled by a factor $1/3$ to the center of each side of $T_0$; and so on. In this way, the open set $A=\bigcup_{h=0}^\infty T_h$ has the well-known von Koch curve $K$ as its topological boundary. If we set
  \[
  v(z)=\min\Big\{\frac{1}2\,,\dist(z,K)\Big\}\,,\qquad z\in\R^2\,,
  \]
  then $v$ is a Lipschitz function on $\R^2$ with $G=\{0<v<1\}=\R^2\setminus K$. Notice that
  \[
  K=\{v^\wedge=0\}=\{v=0\}\subset G^{(1)}\,,\qquad\{v^\vee=1\}=\emptyset,
  \]
  that is $G^{(1)}\cap\{v^\wedge=0\}\cap\{v^\vee=1\}=K$, and thus it is not countably $\H^1$-rectifiable. (Indeed, the Hausdorff dimension of $K$ is equal to $\log(4)/\log(3)$.) In particular, given a Borel partition $\{G_+,G_-\}$ of $G$ we cannot expect the set
  \[
  G^{(1)}\cap\pae G_+\cap\pae G_-\cap\Big(\{v^\wedge=0\}\cup\{v^\vee=1\}\Big)\subset K\,,
  \]
  to possess any rectifiability property. Notice also that, in this example, $K=\{v^\wedge=0\}$ essentially disconnects $\{0<v<1\}$, as it is seen by considering the non-trivial Borel partition $\{G_+,G_-\}$ of $G$ defined by $G_+=A$ and $G_-=\R^2\setminus\ov{A}$. (Indeed, we easily find that  $\pae G_+=\pae G_-\subset K$.) Also, by Theorem \ref{thm gauss}, we expect rigidity to fail. A counterexample to rigidity is obtained by setting
  \[
  E=\Big(F\cap(G_+\times\R)\Big)\cup \Big(g(F)\cap(G_-\times\R)\Big)\,.
  \]
  The fact that $P_\g(E)=P_\g(F)$ descends from the proof of the implication $(i)\Rightarrow (ii)$ that is presented in section \ref{section gauss i implica ii}.}
\end{example}

\subsection{Proof of Theorem \ref{thm gauss}, (i) implies (ii)}\label{section gauss i implica ii}  In this section we present the proof of the implication $(i)\Rightarrow(ii)$ in Theorem \ref{thm gauss}. Let us recall the following general relation for essential boundaries
  \begin{equation}
    \label{AcapB}
      \pae (A\cap B)\cap B^{(1)}=(\pae A)\cap B^{(1)}\,,
  \end{equation}
that holds true for every pair of Lebesgue measurable sets $A, B\subset\R^n$.

\begin{proof}[Proof of Theorem \ref{thm gauss}, (i) implies (ii)]
  {\it Overview}\,: We shall prove that if (ii) fails then (i) fails. Precisely, let us assume the existence of a non-trivial Borel partition $\{G_+,G_-\}$ of $G=\{0<v<1\}$, such that
  \begin{equation}
    \label{mentine}
      \H^{n-2}\Big(\Big(G^{(1)}\cap\pae G_+\cap\pae G_-\Big)\setminus\Big(\{v^\wedge=0\}\cup\{v^\vee=1\}\Big)\Big)=0\,.
  \end{equation}
  We set $G_1=\{v=1\}$, $G_0=\{v=0\}$, and then consider the Borel set
  \[
  E=\Big(F\cap\Big((G_+\cup G_1)\times\R\Big)\Big)\cup \Big(g(F)\cap(G_-\times\R)\Big)\,.
  \]
  The idea here is that since $E$ is obtained by reflecting $F$ across a region where the sections of $F$ are {\it either negligible or equivalent to $\R$}, then we should have $P_\g(E)=P_\g(F)$; however, since  $\H^{n-1}(G_+)\,\H^{n-1}(G_-)>0$ by assumption, this would imply that both $\H^n(E\Delta F)>0$ and $\H^n(E\Delta g(F))>0$, and thus that $(i)$ fails. In order to prove that $P_\g(E)=P_\g(F)$ we shall first need to prove that $E$ is a set of locally finite perimeter, and then use the information that its reduced boundary is $\H^{n-1}$-equivalent to its essential boundary in order to be able to check that no additional Gaussian perimeter is created in passing from $F$ to $E$.

  \medskip

  \noindent {\it Step one:} In this step we gather some preliminary remarks to the strategy of proof described above. We start by noticing that, if we set for the sake of brevity,
  \[
  G_{1\,0\,+}=G_1\cup G_0\cup G_+\,,\qquad
    G_{1\,0\,-}=G_1\cup G_0\cup G_-\,,
  \]
  then by \eqref{AcapB} and $F\cap (G_{1\,0\,+}\times\R)=E\cap(G_{1\,0\,+}\times\R)$ we find that
  \begin{equation}
    \label{nail1}
    \pae F\cap \Big(G_{1\,0\,+}^{(1)}\times\R\Big)=\pae E\cap \Big(G_{1\,0\,+}^{(1)}\times\R\Big)\,.
  \end{equation}
  Similarly, starting from $g(F)\cap (G_{1\,0\,-}\times\R)=E\cap(G_{1\,0\,-}\times\R)$, we deduce that
  \begin{equation}
    \label{nail2}
    \pae (g(F))\cap \Big(G_{1\,0\,-}^{(1)}\times\R\Big)=\pae E\cap \Big(G_{1\,0\,-}^{(1)}\times\R\Big)\,.
  \end{equation}
  By \eqref{nail1} and \eqref{nail2}, we thus find
  \begin{eqnarray}\label{n1}
    \H^{n-1}_\g\Big(\pae E\cap \Big(G_{1\,0\,+}^{(1)}\times\R\Big)\Big)&=&\H^{n-1}_\g\Big(\pae F\cap \Big(G_{1\,0\,+}^{(1)}\times\R\Big)\Big)\,;
    \\\nonumber
    \H^{n-1}_\g\Big(\pae E\cap \Big(G_{1\,0\,-}^{(1)}\times\R\Big)\Big)&=&\H^{n-1}_\g\Big(\pae g(F)\cap \Big(G_{1\,0\,-}^{(1)}\times\R\Big)\Big)
    \\\label{n2}
    &=&\H^{n-1}_\g\Big(\pae  F\cap \Big(G_{1\,0\,-}^{(1)}\times\R\Big)\Big)\,.
  \end{eqnarray}
  By \eqref{n1} and \eqref{n2}, we are left to understand the situation outside the cylinder of basis $\R^{n-1}\setminus(G_{1\,0\,+}^{(1)}\cup G_{1\,0\,-}^{(1)})$. To this end, let us notice that,
  \[
  G_{1\,0\,+}^{(0)}=G_-^{(1)}\,,\qquad G_{1\,0\,-}^{(0)}=G_+^{(1)}\,,\qquad \pae G_{1\,0\,+}=\pae G_-\,,
  \qquad \pae G_{1\,0\,-}=\pae G_+\,,
  \]
  so that
  \begin{eqnarray}\nonumber
  \R^{n-1}\setminus(G_{1\,0\,+}^{(1)}\cup G_{1\,0\,-}^{(1)})&=&\Big(G_{1\,0\,+}^{(0)}\cup\pae G_{1\,0\,+}\Big)
  \cap \Big(G_{1\,0\,-}^{(0)}\cup\pae G_{1\,0\,-}\Big)
  \\\label{nail4}
  &=&\pae G_+\cap\pae G_-\,.
  \end{eqnarray}
  Let us also notice that, by \eqref{mentine} and \cite[2.10.45]{FedererBOOK},
  \begin{equation}
    \label{nail3}
      \H^{n-1}\Big(\Big[\Big(G^{(1)}\cap\pae G_+\cap\pae G_-\Big)\setminus\Big(\{v^\wedge=0\}\cup\{v^\vee=1\}\Big)\Big]\times\R\Big)=0\,.
  \end{equation}
  (Notice that we cannot apply \eqref{federer 3.2.23} here, since $\pae G_+\cap\pae G_-$ may fail to be countably $\H^{n-2}$-rectifiable; see Example \ref{remark koch}.) By taking into account that $\pae G_\s=(\pae G_\s\cap\pae G)\cup(\pae G_\s\cap G^{(1)})$ for $\s\in\{+,-\}$, we are thus left to understand the situation inside the cylinder $(W_1\cup W_2)\times\R$, where we have set,
  \begin{eqnarray*}
    W_1&=&G^{(1)}\cap\pae G_+\cap\pae G_-\cap\Big(\{v^\wedge=0\}\cup\{v^\vee=1\}\Big)\,,
    \\\nonumber
    W_2&=&\pae G\cap\pae G_+\cap\pae G_-\,.
  \end{eqnarray*}
  In fact, by taking into account that
  \begin{eqnarray*}
  \pae G&\subset&\{z\in\R^{n-1}:\theta^*(\{v=0\},z)>0\}\cup\{z\in\R^{n-1}:\theta^*(\{v=1\},z)>0\}
  \\
  &\subset&\{v^\wedge=0\}\cup\{v^\vee=1\}\,,
  \end{eqnarray*}
  we find
  \begin{eqnarray*}
    W_2=\pae G\cap\pae G_+\cap\pae G_-\cap\Big(\{v^\wedge=0\}\cup\{v^\vee=1\}\Big)\,,
  \end{eqnarray*}
  so that
  \begin{equation}
    \label{W}
     W_1\cup W_2=\pae G_+\cap\pae G_-\cap\Big(\{v^\wedge=0\}\cup\{v^\vee=1\}\Big)\,.
  \end{equation}

  \medskip

  \noindent {\it Step two}\,: We show that $E$ and $F$ have no essential boundary above $\{v^\vee=0\}\cup\{v^\wedge=1\}$. Indeed, we are going to prove
  \begin{eqnarray}\label{antani1}
  &&\{v^\vee=0\}\times\R\subset E^{(0)}\cap F^{(0)}\,,
  \\\label{antani2}
  &&\{v^\wedge=1\}\times\R\subset E^{(1)}\cap F^{(1)}\,,
  \end{eqnarray}
  thus deducing that
  \begin{eqnarray}
    \label{n4}
    &&\H^{n-1}_\g\Big(\pae F\cap\Big(\{v^\vee=0\}\times\R\Big)\Big)=     \H^{n-1}_\g\Big(\pae E\cap\Big(\{v^\vee=0\}\times\R\Big)\Big)=0\,,
    \\
    \label{n5}
    &&\H^{n-1}_\g\Big(\pae F\cap\Big(\{v^\wedge=1\}\times\R\Big)\Big)=\H^{n-1}_\g\Big(\pae E\cap\Big(\{v^\wedge=1\}\times\R\Big)\Big)=0\,.
  \end{eqnarray}
  Let us show for example that if $z\in\{v^\vee=0\}$, then $(z,s)\in E^{(0)}$ for every $s\in\R$. Indeed, if $s\in\R$ and $r<1$, then
  \begin{eqnarray*}
    \H^n(E\cap\C_{(z,s),r})&=&2r\,\H^{n-1}(\D_{z,r}\cap G_1)
    +\int_{s-r}^{s+r}\,\H^{n-1}\Big(\D_{z,r}\cap G_+\cap \{f<t\}\Big)\,dt
    \\
    &&+\int_{s-r}^{s+r}\,\H^{n-1}\Big(\D_{z,r}\cap G_-\cap \{f<-t\}\Big)\,dt
    \\
    &\le& 2r\,\H^{n-1}(\D_{z,r}\cap G_1)+2r\,\H^{n-1}\Big(\D_{z,r}\cap \{f<|s|+1\}\Big)=o(r^n)\,,
  \end{eqnarray*}
  where in the last identity we have used the assumption that $v^\vee(z)=0$ (and thus $f^\wedge(z)=+\infty$) to deduce that $\theta(\{f<|s|+1\},z)=0$. This proves \eqref{antani1}, and \eqref{antani2} follows analogously.

  \medskip

  \noindent {\it Step three}\,: We show that $E$ is of locally finite perimeter. To this end, by taking into account step one and step two, it suffices to prove that
  \begin{equation}
    \label{bon}
  \H^{n-1}_\g(\pae E\cap(\S_1\times\R))<\infty\,,
  \end{equation}
  where we have set
  \[
  \S_1=\pae G_+\cap\pae G_-\cap\Big(\{0=v^\wedge<v^\vee\}\cup\{v^\wedge<v^\vee=1\}\Big)\,.
  \]
  We now claim that, if $z\in\{0=v^\wedge<v^\vee\}\cup\{v^\wedge<v^\vee=1\}$, then
  \begin{equation}
    \label{pisainB}
    (\pae E)_z\subset_{\H^1} (\pae F)_z\cup\,(\pae g(F))_z\,.
  \end{equation}
  Indeed, on the one hand, by \eqref{inclusioni tizio}, we have that
  \begin{align}\label{pisa3}
    &(\pae F)_z=_{\H^1}[f^\wedge(z),\infty)\,,&  \forall z\in \{0=v^\wedge<v^\vee\}\,,
    \\\label{pisa4}
    &(\pae F)_z=_{\H^1}(-\infty,f^\vee(z)]\,,& \forall z\in\{v^\wedge<v^\vee=1\}\,;
  \end{align}
  on the other hand, we also have, {\it for every $z\in\R^{n-1}$},
  \begin{eqnarray}\label{pisa1}
    (\pae E)_z&\subset&(-\infty,-f^\wedge(z)]\cup [f^\wedge(z),\infty)\,,
    \\\label{pisa2}
    (\pae E)_z&\subset&(-\infty,f^\vee(z)]\cup [-f^\vee(z),\infty)\,;
  \end{eqnarray}
  see Figure \ref{fig federico}.
  \begin{figure}
    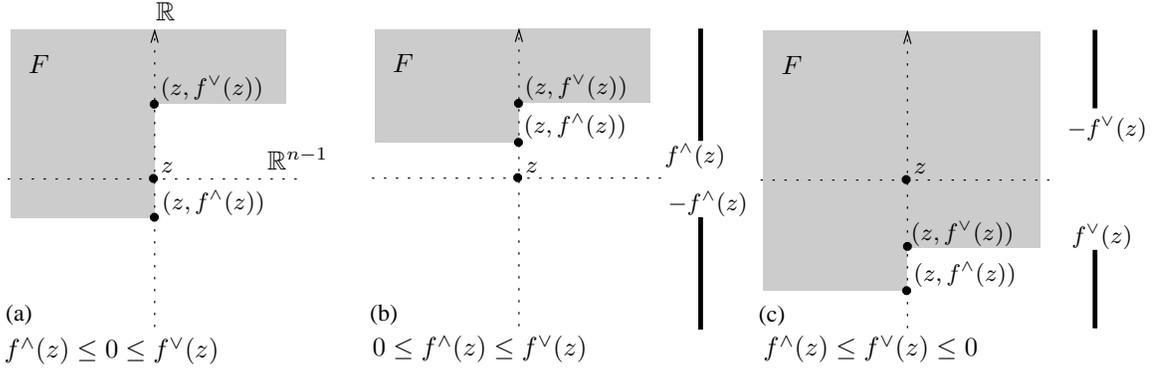\caption{{\small The three cases one has to consider in describing $(\pae E)_z$. Notice that, in case (a), both inclusions \eqref{pisa1} and \eqref{pisa2} are trivial; in case (b), \eqref{pisa2} is trivial, and \eqref{pisa1} carries all the useful information; finally, in case (c), \eqref{pisa1} is trivial, and \eqref{pisa2} is not.}}\label{fig federico}
  \end{figure}
  Let us show, for example, the validity of \eqref{pisa1}: if $f^\wedge(z)\le0$, then inclusion is trivial; if we thus assume that $f^\wedge(z)>0$, then we have $v^\vee(z)<1/2$, thus that
  \[
  0=\theta(\{v>2/3\},z)\ge\theta(G_1,z)\,;
  \]
  that is, $z\in G_1^{(0)}$. Hence,
  \begin{eqnarray*}
    \H^n(E\cap\C_{(z,t),r})&=&2r\H^{n-1}(G_1\cap\D_{z,r})+\int_{t-r}^{t+r}\H^{n-1}(G_+\cap\{f<s\}\cap\D_{z,r})\,ds
    \\
    &&+\int_{t-r}^{t+r}\H^{n-1}(G_-\cap\{f<-s\}\cap\D_{z,r})\,ds
    \\
    &\le&o(r^n)+2r\,\H^{n-1}(\{f<|t|+r\}\cap\D_{z,r})\,;
  \end{eqnarray*}
  therefore, if $t\in(-f^\wedge(z),f^\wedge(z))$ and $r<r_*$ for a suitable value of $r_*$, then we have
  \[
  \H^n(E\cap\C_{(z,t),r})\le o(r^n)+2r\,\H^{n-1}(\{f<|t|+r_*\}\cap\D_{z,r})=o(r^n)\,,
  \]
  that is, $(z,t)\in E^{(0)}$; in other words,
  \[
  (-f^\wedge(z),f^\wedge(z))\subset (E^{(0)})_z\subset\R\setminus(\pae E)_z\,,
  \]
  that is \eqref{pisa1}. The proof of \eqref{pisa2} is analogous; by taking into account \eqref{pisa3}, \eqref{pisa4}, \eqref{pisa1}, and \eqref{pisa2}, we thus find \eqref{pisainB}, which in particular gives
  \begin{equation}
    \label{pisa}
      \H^{n-1}_\g\Big(\pae E\cap\Big(\S_1\times\R\Big)\Big)\le 2\, \H^{n-1}_\g\Big(\pae F\cap\Big(\S_1\times\R\Big)\Big)\,,
  \end{equation}
  and thus proves \eqref{bon}. By \eqref{n1}, \eqref{n2}, \eqref{nail4}, \eqref{nail3}, \eqref{W}, \eqref{n4}, \eqref{n5} and \eqref{bon}, we thus find $\H^{n-1}_\g(\pae E)<\infty$. Hence, by Federer's criterion, $E$ is of locally finite perimeter.

  \medskip

  \noindent {\it Step four:} We have proved so far that $E$ is a set of locally finite perimeter with
  \[
  P_\g(E;(\R^{n-1}\setminus\S_1)\times\R)=P_\g(F;(\R^{n-1}\setminus\S_1)\times\R)
  \]
  Since $P_\g(E;W\times\R)\ge P_\g(F;W\times\R)$ for every Borel set $W\subset\R^{n-1}$, we only need to show
  \begin{equation}
    \label{superfilippo2}
    P_\g(E;\S_1\times\R)\le P_\g(F;\S_1\times\R)\,.
  \end{equation}
  By Federer's theorem, $\H^{n-1}(\pae E\setminus\pa^JE)=0$, and, moreover, by Proposition \ref{proposition fine bv} we have that $\H^{n-2}(S_f\setminus J_f)=0$ (so that $\H^{n-1}((S_f\setminus J_f)\times\R)=0$). Since $\{v^\wedge=0\}=\{f^\vee=\infty\}$ and $\{v^\vee=1\}=\{f^\wedge=-\infty\}$, we conclude that \eqref{superfilippo2} follows by
  \begin{equation}
    \label{superfilippo3}
    \H^{n-1}_\g(\pa^JE\cap(\S_2\times\R))\le\H^{n-1}_\g(\pae F\cap(\S_2\times\R))\,,
  \end{equation}
  where
  \[
  \S_2=\pae G_+\cap\pae G_-\cap J_f\cap\Big(\Big\{-\infty<f^\wedge<f^\vee=\infty\Big\}\cup\Big\{-\infty=f^\wedge<f^\vee<\infty\Big\}\Big)\,.
  \]
  We now turn to the proof of \eqref{superfilippo3}, and thus complete the proof of the implication $(ii)\Rightarrow(i)$. To this end, we pick
  \[
  z\in J_f\cap\Big(\Big\{-\infty<f^\wedge<f^\vee=\infty\Big\}\cup\Big\{-\infty=f^\wedge<f^\vee<\infty\Big\}\Big)
  \]
  and show that either $(\pa^JE)_z \subset_{\H^1} (\pa^J F)_z$ or $(\pa^JE)_z \subset_{\H^1} g ( (\pa^J F)_z)$. In fact, by symmetry, we only have to consider the case
  \begin{equation}
    \label{superfilippo4}
      z\in J_f \cap \{ - \infty < f^\wedge <f^\vee=\infty\}\,.
  \end{equation}
  Under assumption \eqref{superfilippo4},
  \begin{figure}
    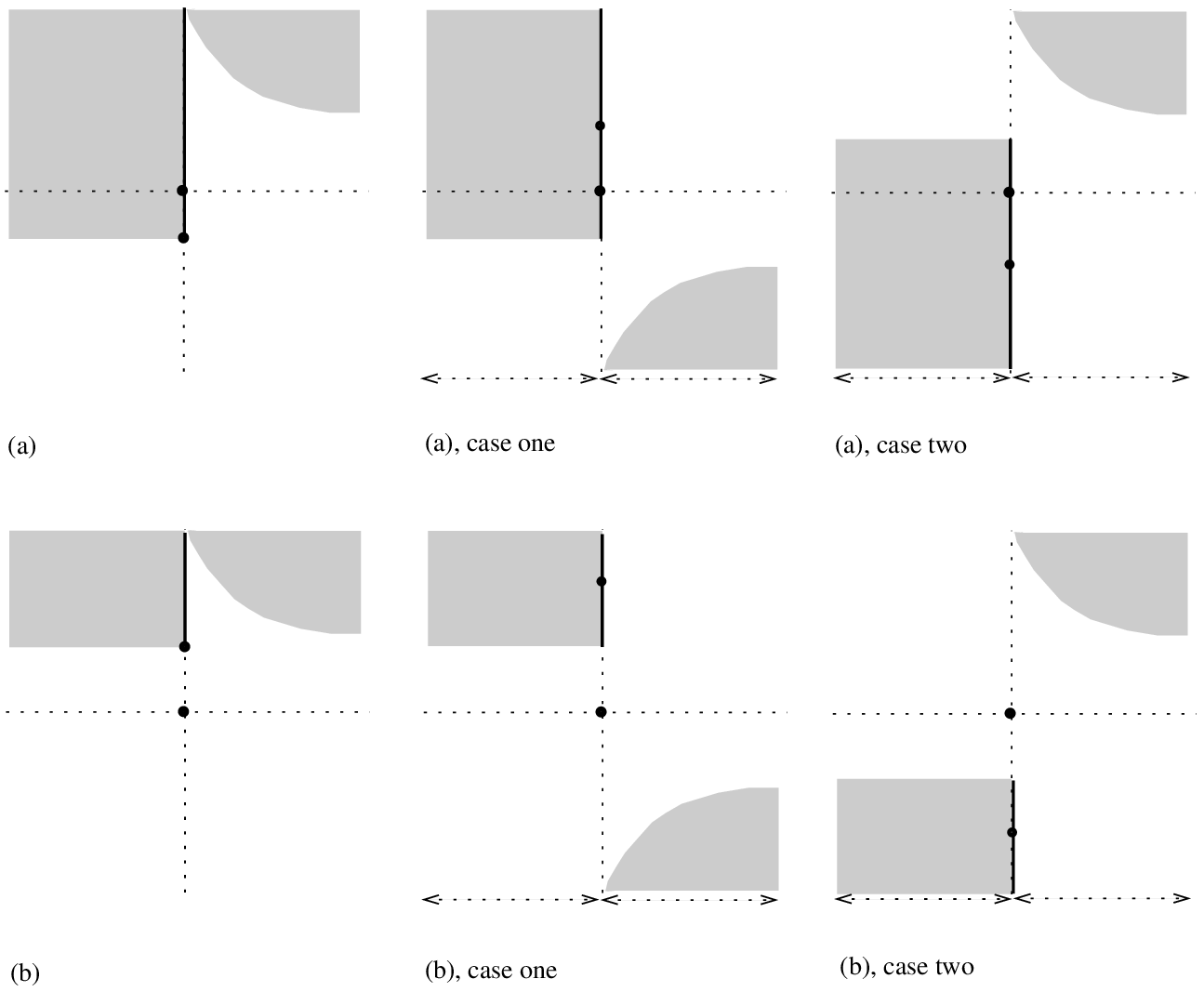\caption{{\small The situation in the proof of \eqref{mlk1} and \eqref{mlk2}. If $f^\wedge(z)\le 0$, then \eqref{centro1} shows that $(f^\wedge(z),-f^\wedge(z))$ is contained both in $(\pae F)_z$ and $(\pae E)_z$. Moreover, if $f^\wedge(z)\le 0$ and we are in case one, then, see \eqref{subcase1}, there exists $t_0>-f^\wedge(z)$ such that $(z,t_0)\in\pa^JE$, $(\pae E)_z$ and $(\pae F)_z$ are both $\H^1$-equivalent to $(f^\wedge(z),\infty)$, and \eqref{mlk1} holds true. Finally, if $f^\wedge(z)\le0$ and we are in case two, then, see \eqref{subcase2}, there exists $t_0<f^\wedge(z)$ such that $(z,t_0)\in\pa^JE$, $(\pae E)_z$ and $g((\pae F)_z)$ are both $\H^1$ equivalent to $(-\infty,f^\wedge(z))$, and thus \eqref{mlk2} holds true. Similar remarks apply when $f^\wedge(z)>0$.}}\label{fig federico2}
  \end{figure}
  we thus want to show that
    \begin{eqnarray}\label{mlk1}
    &&\mbox{either}\qquad (\pa^JE)_z \subset_{\H^1} (\pa^J F)_z =_{\H^1}\,(f^\wedge(z),\infty)\,,
    \\\label{mlk2}
    &&\mbox{or}\hspace{0.6cm}\qquad (\pa^JE)_z \subset_{\H^1} g \Big( (\pa^J F)_z \Big) =_{\H^1}\,(-\infty, - f^\wedge(z))\,.
  \end{eqnarray}
  We first notice that, by Lemma \ref{lemma fico}, there exists $\nu\in S^{n-1} \cap \R^{n-1}$ such that
  \begin{eqnarray}
  \label{f uno}
  \{f<s\}_{z,r}\toloc H_{z,\nu}^+\,,\qquad \forall\,s>f^\wedge(z)\,,
  \\
  \label{f due}
  \{f < s\}_{z,r}\toloc \emptyset\,,\qquad \forall\,s<f^\wedge(z)\,.
  \end{eqnarray}
  Moreover, we have the inclusions,
  \begin{eqnarray} \label{centro1}
  (\partial^J E)_z\cap\Big(f^\wedge(z),-f^\wedge(z)\Big)&\subset_{\H^1}& (\partial^J F)_z\,,\qquad\mbox{if $f^\wedge(z)\le 0$}\,,
  \\\label{centro2}
  (\partial^J E)_z\cap\Big(-f^\wedge(z),f^\wedge(z)\Big)&\subset_{\H^1}& \emptyset\,,\hspace{1cm}\qquad\mbox{if $f^\wedge(z)>0$}\,,
  \end{eqnarray}
  that follow by \eqref{pisa3} and \eqref{pisa1}. We now divide our argument into two cases.

  \medskip

  \noindent {\it Case one}\,: assuming that there exists $t_0\in(\pa^J E)_z$ with $t_0>|f^\wedge(z)|$ we show that
  \begin{equation}
    \label{subcase1}
      \Big\{t\in(\pa^JE)_z:|t|>|f^\wedge(z)|\Big\}=\Big(|f^\wedge(z)|,\infty\Big)\,.
  \end{equation}

  \noindent {\it Case two}\,: assuming that there exists $t_0\in(\pa^J E)_z$ with $t_0<-|f^\wedge(z)|$ we show that
  \begin{equation}
    \label{subcase2}
  \Big\{t\in(\pa^JE)_z:|t|>|f^\wedge(z)|\Big\}=\Big(-\infty,-|f^\wedge(z)|\Big)\,.
  \end{equation}

  \medskip

  \noindent Before entering into the proof of the two cases, let us notice how they allow to complete the proof of the theorem (see also Figure \ref{fig federico2}.) Indeed, if none of the two cases holds true, this means that $(\pa^JE)_z\subset_{\H^1}(-|f^\wedge(z)|,|f^\wedge(z)|)$,
  and then the validity of either \eqref{mlk1} or \eqref{mlk2} follows by \eqref{centro1} and \eqref{centro2}. (Just notice that when $f^\wedge(z)\le 0$, then $(-|f^\wedge(z)|,|f^\wedge(z)|)\subset(\pa^JF)_z\cap g(\pa^JF)_z$.) Similarly, if we are in the first case, and $f^\wedge(z)>0$, then \eqref{mlk1} follows by combining \eqref{centro2} with \eqref{subcase1}; if we are in the first case and $f^\wedge(z)\le 0$, then \eqref{mlk1} follows by \eqref{centro1} and \eqref{subcase1}; finally, if we are in the second case then \eqref{mlk2} holds true by combining \eqref{centro1}, \eqref{centro2}, and \eqref{subcase2}. We prove \eqref{subcase1} and \eqref{subcase2} in the next step.

  \medskip

  \noindent {\it Step five}\,: We assume to be in the first case, and prove \eqref{subcase1}. Let us first prove that
  \begin{equation}
  \label{f tre}
  \H^{n-1}(\D_{z,r}\cap G_{1\,+}\cap H^+_{z,\nu})=\frac{\om_{n-1}\,r^{n-1}}2+o(r^{n-1})\,,
\end{equation}
and thus, clearly, that
\begin{equation}
  \label{f quattro}
\H^{n-1}(\D_{z,r}\cap G_{0\,-}\cap H^+_{z,\nu})=o(r^{n-1})\,,
\end{equation}
where we have set $G_{1\,+}=G_{+}\cup G_1$ and $G_{0\,-}=G_{-}\cup G_0$. To prove \eqref{f tre}, we pick $t_1$ and $t_2$ such that $|f^\wedge(z)|<t_1<t_0<t_2$. Since $(z,t_0)\in\pa^JE$, and since every half-space $H$ with $x\in\pa H$ cuts $\C_{x,r}$ into two halves of equal volume, we find that
\begin{eqnarray*}
  &&\frac{\H^n(\C_{(z,t_0),r})}2+o(r^n)=\H^n(E\cap\C_{(z,t_0),r})
  \\
  &&\hspace{1.2cm}=\int_{t_0-r}^{t_0+r}\H^{n-1}(\D_{z,r}\cap G_{1\,+}\cap\{f<s\})
  +\H^{n-1}(\D_{z,r}\cap G_{-}\cap\{f<-s\})\,ds
   \\
  &&\hspace{1.2cm}\le2r\,\H^{n-1}(\D_{z,r}\cap G_{1\,+}\cap \{f<t_2\})\,ds+o(r^n)\,,
\end{eqnarray*}
where in the last identity we have used \eqref{f due} with $s=-t_1<-|f^\wedge(z)|\le f^\wedge(z)$. By applying \eqref{f uno} with $s=t_2>|f^\wedge(z)|\ge f^\wedge(z)$, and since $\H^{n-1}(\C_r)=2\,\om_{n-1}r^n$, we find
\begin{eqnarray*}
  \om_{n-1}r^n+o(r^n)\le 2r\,\H^{n-1}(\D_{z,r}\cap G_{1\,+}\cap H_{z,\nu}^+)+o(r^n)\,.
\end{eqnarray*}
that is
\begin{eqnarray*}
  \frac{\om_{n-1}r^{n-1}}2+o(r^{n-1})\le\H^{n-1}(\D_{z,r}\cap G_{1\,+}\cap H_{z,\nu}^+)\le\frac{\om_{n-1}r^{n-1}}2\,.
\end{eqnarray*}
This proves \eqref{f tre}, and thus \eqref{f quattro}. We now pick $t>|f^\wedge(z)|$, we now choose $t_1$ and $t_2$ to be such that $|f^\wedge(z)|<t_1<t<t_2$, and then notice that
\begin{eqnarray*}
  \H^n((E\Delta H_{(z,t),(\nu,0)}^+)\cap\C_{(z,t),r})
  &=&\int_{t-r}^{t+r}\H^{n-1}(\D_{z,r}\cap G_{1\,+}\cap\{f<s\}\cap H_{z,\nu}^-)\,ds
  \\
  &&+\int_{t-r}^{t+r}\H^{n-1}(\D_{z,r}\cap G_{1\,+}\cap\{f\ge s\}\cap H_{z,\nu}^+)\,ds
  \\
  &&+\int_{t-r}^{t+r}\H^{n-1}(\D_{z,r}\cap G_{-}\cap\{f<-s\}\cap H_{z,\nu}^-)\,ds
   \\
  &&+\int_{t-r}^{t+r}\H^{n-1}(\D_{z,r}\cap G_{-}\cap\{f\ge-s\}\cap H_{z,\nu}^+)\,ds\,,
\end{eqnarray*}
so that, $\H^n((E\Delta H_{(z,t),(\nu,0)}^+)\cap\C_{(z,t),r})\le 2r\,(I_1+I_2+I_3+I_4)$ where
\begin{eqnarray*}
I_1&=&\H^{n-1}(\D_{z,r}\cap G_{1\,+}\cap\{f<t_2\}\cap H_{z,\nu}^-)\,,
\\
I_2&=&\H^{n-1}(\D_{z,r}\cap G_{1\,+}\cap\{f \ge t_1\}\cap H_{z,\nu}^+)\,,
\\
I_3&=&\H^{n-1}(\D_{z,r}\cap G_{-}\cap\{f<-t_1\}\cap H_{z,\nu}^-)\,,
\\
I_4&=&\H^{n-1}(\D_{z,r}\cap G_{-}\cap\{f\ge-t_2\}\cap H_{z,\nu}^+)\,.
\end{eqnarray*}
We see that $I_1= I_2 = o(r^{n-1})$ by \eqref{f uno}, while $I_3=o(r^{n-1})$ by \eqref{f due}, and $I_4=o(r^{n-1})$ by \eqref{f quattro}.
We have thus proved that
\[
\Big(|f^\wedge(z)|,\infty\Big) \subset (\pa^JE)_z\,.
\]
In order to conclude the proof of \eqref{subcase1} we will now prove that
\[
\Big(-\infty,-|f^\wedge(z)|\Big)\subset E^{(0)}\,.
\]
Indeed, let us pick $t<-|f^\wedge(z)|$. This time we set $t_1$ and $t_2$ to be such that $t_1<t<t_2<-|f^\wedge(z)|$. In this way, by arguing as above, and by also recalling that $z \in G_1^{(0)}$, we find
\begin{eqnarray*}
  \H^n(E \cap\C_{(z,t),r})
  &\leq&2 r \, \H^{n-1}(\D_{z,r}\cap G_{+}\cap\{f<t_2\})
  \\
  &&+2r \, \H^{n-1}(\D_{z,r}\cap G_{-}\cap\{f<-t_1\})\,.
\end{eqnarray*}
where the first term is $o(r^n)$ by \eqref{f due}. By \eqref{f uno} we thus find
\begin{eqnarray*}
  \H^n(E \cap\C_{(z,t),r})=o(r^n)+2r \, \H^{n-1}(\D_{z,r}\cap G_{-}\cap H_{z,\nu}^+)=o(r^n)\,,
\end{eqnarray*}
where the last identity follows by \eqref{f quattro}. Hence $(z,t) \in E^{(0)}$, as claimed, and the proof of \eqref{subcase1} is completed. In order to prove \eqref{subcase2}, we notice that the existence of $t_0<-|f^\wedge(z)|$ such that $(z,t_0)\in\pa^JE$, implies
\begin{equation}
  \label{f quattroa}
\H^{n-1}(\D_{z,r}\cap G_{1\,+}\cap H^+_{z,\nu})=o(r^{n-1})\,.
\end{equation}
The proof of \eqref{subcase2} is then analogous to that of \eqref{subcase1}, with \eqref{f quattroa} in place of \eqref{f quattro}.
\end{proof}

\subsection{Proof of Theorem \ref{thm gauss} and Theorem \ref{thm gauss R2}}\label{section gauss conclusion} We finally complete the proof of our two main results.

\begin{proof}
  [Proof of Theorem \ref{thm gauss}] The equivalence of (i) and (ii) is proved in sections \ref{section gauss ii implica i} and \ref{section gauss i implica ii}.
\end{proof}

\begin{proof}
  [Proof of Theorem \ref{thm gauss R2}] {\it Step one}\,: We show that if a Borel set $G\subset\R$ is essentially connected, then $G^{(1)}$ is an interval. Indeed, let us prove that, if $a,b\in G^{(1)}$ with $a<b$ and $c\in(a,b)$, then $c\in G^{(1)}$. To see this, we set $G_+=G\cap(c,\infty)$, $G_-=G\cap(-\infty,c)$, so that $\{G_+,G_-\}$ is a Borel partition of $G$ modulo $\H^1$. In fact, $\H^1(G_+)\H^1(G_-)>0$. Indeed, should $\H^1(G_+)=0$, then we would have $(G_+)^{(1)}=\emptyset$, and thus
  \[
  b\in G^{(1)}\cap(c,\infty)^{(1)}\subset\Big(G\cap(c,\infty)\Big)^{(1)}=(G_+)^{(1)}=\emptyset\,,
  \]
  a contradiction. Since $G$ is essentially connected, we find
  \begin{equation}
    \label{R2}
      \H^0(G^{(1)}\cap\pae G_+\cap\pae G_-)>0\,.
  \end{equation}
  Since $G^{(1)}\cap\pae G_+=G^{(1)}\cap\{c\}$ and $G^{(1)}\cap\pae G_-=G^{(1)}\cap\{c\}$, \eqref{R2} gives $c\in G^{(1)}$.

  \medskip

  \noindent {\it Step two}\,: If $\{v^\wedge=0\}\cup\{v^\vee=1\}$ does not essentially disconnect $\{0<v<1\}$, then, in particular, $\{0<v<1\}$ is essentially connected, and thus $\H^1$-equivalent to an open interval $I$ by step one. Let now $c\in I$, and assume that $v^\wedge(c)=0$. Since $\{c\}$ (thus $\{v^\wedge=0\}$) essentially disconnects $I$, by Remark \ref{remark essential connected} we find that $\{v^\wedge=0\}$ essentially disconnects $\{0<v<1\}$, a contradiction. Therefore, $v^\wedge>0$ on $I$. We similarly see that $v^\vee<1$ on $I$. This shows that assumption (ii) in Theorem \ref{thm gauss} implies assumption (ii) in Theorem \ref{thm gauss R2}. Since the reverse implication is trivial, we are done.
\end{proof}

\section{Some further conditions for rigidity}\label{section hey} As noticed in Remark \ref{remark sufficient}, a natural question is whether it is possible to formulate sufficient conditions for rigidity in terms of suitable connectedness properties of $F[v]$. Referring readers to the remark for a list of examples and possible conditions, we prove here two results, that provide simple sufficient conditions for rigidity.

\begin{theorem}\label{thm pino}
  If $v:\R^{n-1}\to[0,1]$ is Lebesgue measurable and such that $P_\g(F[v])<\infty$, and if there exists a sequence $t_h\to 0$ as $h\to\infty$ such that, for every $h\in\N$,
  \begin{equation}\label{kkk}
  \mbox{$F[v]\cap\Big(\{t_h<v<1-t_h\}\times\R\Big)$ is essentially connected in $\R^n$}\,,
  \end{equation}
  then $E\in\M(v)$ if and only if $\H^n(E\Delta F[v])=0$ or $\H^n(E\Delta g(F[v]))=0$.
\end{theorem}

\begin{proof}
  We notice that in the proof of (ii) implies (i) in Theorem \ref{thm gauss}, assumption (ii) was used only to guarantee the validity of \eqref{h2}, that in turn was used in step five of that proof to deduce that $\H^{n-2}(\S_k)>0$. Thus, in order to prove that \eqref{kkk} implies rigidity, it will suffice to show that it implies $\H^{n-2}(\S_k)>0$ for $k$ large enough. Let us now set
  \[
  G_h=\{t_h<v<1-t_h\}\,,\qquad F_h=F\cap (G_h\times\R)\,,\qquad h\in\N\,.
  \]
  If we set $G_{h,+}=G_+\cap G_h$ and $G_{h,-}=G_-\cap G_h$, then $\H^{n-1}(G_{h,\pm})\to\H^{n-1}(G_{\pm})$ as $h\to\infty$. Hence, $\H^{n-1}(G_{h,+})\H^{n-1}(G_{h,-})>0$ for $h$ large enough, and, correspondingly, the sets
  \[
  F_{h,+}=F\cap(G_{h,+}\times\R)\,,\qquad F_{h,-}=F\cap(G_{h,-}\times\R)\,,
  \]
  define a non-trivial Lebesgue measurable partition $\{F_{h,+},F_{h,-}\}$ of $F_h$. By \eqref{kkk},
\begin{equation}
  \label{positivo1}
  \H^{n-1}\Big(\pae F_{h,+}\cap\pae F_{h,-}\cap F_h^{(1)}\Big)>0\,,
\end{equation}
for $h$ large enough. Let us now set
\[
\Lambda_h=\p\Big(\pae F_{h,+}\cap\pae F_{h,-}\cap F_h^{(1)}\Big)\,,\qquad\forall h\in\N\,.
\]
If $\H^{n-2}(\Lambda_h)<\infty$, then, by \cite[2.10.45]{FedererBOOK}, for every $R>0$ we have
\begin{eqnarray*}
\H^{n-2}(\Lambda_h)\L^1((-R,R))&\ge& c(n)\,\H^{n-1}(\Lambda_h\times(-R,R))
\\
&\ge& c(n)\,\H^{n-1}\Big(\pae F_{h,+}\cap\pae F_{h,-}\cap F_h^{(1)}\cap\{|\q x|<R\}\Big)\,,
\end{eqnarray*}
so that, by \eqref{positivo1}, $\H^{n-2}(\Lambda_h)>0$ for every $h$ large enough. We now claim that, given $h\in\N$ there exists $k_h\in\N$ such that
\begin{equation}
  \label{inclusion0}
  \Lambda_h\subset\Sigma_k\,,\qquad\forall k\ge k_h\,;
\end{equation}
this will conclude the proof. To show \eqref{inclusion0}, we start noticing that
\begin{eqnarray*}
  z\in G_+^{(0)}&&\Rightarrow\qquad  z\in G_{h,+}^{(0)}
  \\
  &&\Rightarrow\qquad(z,s)\in (G_{h,+}\times\R)^{(0)}\,,\qquad\hspace{0.95cm}\forall s\in\R\,,
  \\
  &&\Rightarrow\qquad(z,s)\in [F\cap(G_{h,+}\times\R)]^{(0)}\,,\qquad\forall s\in\R\,,
  \\
  &&\Rightarrow\qquad z\notin \p(\pae F_{h,+})\,;
\end{eqnarray*}
similarly, being $G_+$ and $G_-$ disjoint, $z\in G_+^{(1)}$ implies $z\in G_{-}^{(0)}$, and thus $z\notin \p(\pae F_{h,-})$. We have thus proved so far that
\begin{equation}
  \label{inclusion1}
  \Lambda_h\subset\p\Big(\pae F_{h,+}\cap\pae F_{h,-}\Big)\subset\pae G_+\cap\pae G_-\,,\qquad\forall h\in\N\,.
\end{equation}
We now notice that
\begin{eqnarray}\nonumber
  G_h^{(1)}&\subset&\{v>t_h\}^{(1)}\cap\{v<1-t_h\}^{(1)}
  \\\nonumber
  \mbox{{\small (by \eqref{dens2} and \eqref{dens3})}}\qquad&\subset&\{v^\wedge\ge t_h\}\cap\{v^\vee\le 1-t_h\}
  \\\label{inclusion2}
  \mbox{{\small (by \eqref{continuous and decreasing})}}\qquad&\subset&\{f^\vee\le \Psi(t_h)\}\cap\{f^\wedge\ge \Psi(1-t_h)\}\,.
\end{eqnarray}
Hence, if $x\in F_h^{(1)}$, then
$x\in (G_h\times\R)^{(1)}$, and thus $\p x=z\in G_h^{(1)}$, so that, by \eqref{inclusion2},
\begin{equation}
  \label{inclusion3}
  \Lambda_h\subset G_h^{(1)}\subset\Big\{f^\vee\le\Psi(t_h)\Big\}\cap\Big\{f^\wedge\ge\Psi(1-t_h)\Big\}\,,\qquad\forall h\in\N\,.
\end{equation}
By combining \eqref{inclusion1}, \eqref{inclusion3}, and the definition of $\Sigma_k$, we thus come to prove \eqref{inclusion0}, provided we choose $k_h$ such that $k_h>\Psi(t_h)$ and $-k_h<\Psi(1-t_h)$. This completes the proof.
\end{proof}

\begin{theorem}\label{thm gino}
  If $v:\R\to[0,1]$ is Lebesgue measurable with $P_\g(F[v])<\infty$, and both $F[v]$ and $\R^2\setminus F[v]$ are indecomposable sets, then $E\in\M(v)$ if and only if $\H^2(E\Delta F[v])=0$ or $\H^2(E\Delta g(F[v]))=0$.
\end{theorem}

\begin{proof}
  {\it Step one:} We show that, if $F=F[v]$ is indecomposable in $\R^2$ and $v^\wedge(c)=0$, then
  \begin{equation}
    \label{johell1}
    \H^2(F\cap((c,\infty)\times\R))\,\H^2(F\cap((-\infty,c)\times\R))=0\,.
  \end{equation}
  Indeed, let us assume this is not the case, and set $F_+=F\cap((c,\infty)\times\R)$ and $F_-=F\cap((-\infty,c)\times\R)$. We claim that $\{F_+,F_-\}$ is a non-trivial partition of $F$ by sets of locally finite perimeter with the property that
  \begin{equation}
    \label{johell2}
      F^{(1)}\cap\pae F_+\cap\pae F_-=\emptyset\,,
  \end{equation}
  against the indecomposability of $F$. To show that \eqref{johell2} holds true, let us notice that since $F_+$ and $F_-$ are disjoint subsets of $F$ whose union is $F$, we have
  \[
  F^{(1)}\cap\pae F_+\cap\pae F_-=F^{(1)}\cap\pae F_+=F^{(1)}\cap(\{c\}\times\R)\,.
  \]
  However, if $(c,t)\in F^{(1)}$ for some $t\in\R$, then for every $r<1$ we find
  \begin{eqnarray*}
    4\,r^2+o(r^2)&=&\H^2(F\cap\C_{(c,t),r})=\int_{t-r}^{t+r}\H^1(\D_{c,r}\cap\{f<s\})\,ds
    \\
    &\le&2r\,\H^1(\D_{c,r}\cap\{f<t+1\})\,,
  \end{eqnarray*}
  which leads to a contradiction as, by $v^\wedge(c)=0$ (that is, $f^\vee(c)=+\infty$), we have
  \[
  \liminf_{r\to 0^+}\frac{\H^1(\D_{c,r}\cap\{f<t+1\})}{2r}<1\,.
  \]
  This proves \eqref{johell2}, thus our claim.

  \medskip

  \noindent {\it Step two:} By arguing as in the proof of step one, we notice that, if $\R^2\setminus F$ is indecomposable in $\R^2$ and $v^\vee(c)=1$, then
  \begin{equation}
    \label{johell3}
    \H^2(((c,\infty)\times\R)\setminus F)\,\H^2(((-\infty,c)\times\R)\setminus F)=0\,.
  \end{equation}

  \medskip

  \noindent {\it Step three:} We show that, if both $F$ and $\R^2\setminus F$ are indecomposable, then $\{0<v<1\}$ is $\H^1$-equivalent to an open interval. Indeed, let $I$ be the least closed interval that contains $\{0<v<1\}$ modulo $\H^1$. If $\{0<v<1\}$ is not $\H^1$-equivalent to $I$, then there exists $J\subset I\cap(\{v=0\}\cup\{v=1\})$ with $\H^1(J)>0$. In particular, if $\e=\H^1(J)/3$, then there exists $c\in J^{(1)}$ with
  \begin{equation}
    \label{johell4}
      c>\inf\,I+\e\,,\qquad c<\sup I\,-\e\,.
  \end{equation}
  By \eqref{johell4}, and by minimality of $I$, we see that
  \begin{eqnarray}
    \label{johell5}
  \H^2(F\cap((c,\infty)\times\R))\,\H^2(F\cap((-\infty,c)\times\R))>0\,,
  \\
  \label{johell6}
    \H^2(((c,\infty)\times\R)\setminus F)\,\H^2(((-\infty,c)\times\R)\setminus F)>0\,.
  \end{eqnarray}
  By $c\in J^{(1)}$ we find $c\in(\{v=0\}\cup\{v=1\})^{(1)}$, and thus either $\theta^*(\{v=0\},c)>0$ or $\theta^*(\{v=1\},c)>0$; therefore, either $v^\wedge(c)=0$ (but then \eqref{johell5} contradicts \eqref{johell1}) or $v^\vee(c)=1$ (but then \eqref{johell6} contradicts \eqref{johell3}). Hence, $\{0<v<1\}$ is $\H^1$-equivalent to the interval $I$.

  \medskip

  \noindent {\it Step four:} We prove the validity of condition (ii) in Theorem \ref{thm gauss R2} by a simple combination of the first three steps. Hence, rigidity holds true by Theorem \ref{thm gauss R2}.
\end{proof}

\bibliography{references}
\bibliographystyle{is-alpha}

\end{document}